\documentclass[10pt]{article}
\usepackage{amsfonts,amssymb,amsbsy,amsmath,amsthm,enumerate,verbatim}
\usepackage{graphicx}
\usepackage{epsfig}
\input xepsf

\newtheorem{theorem}{Theorem}[section]
\newtheorem{lemma}{Lemma}[section]
\newtheorem{follow}{Corollary}[section]
\newtheorem{pr}{Proposition}[section]
\newtheorem{remark}[theorem]{Remark}

\newcommand{\bel}{\begin{equation} \label}
\newcommand{\ee}{\end{equation}}

\newcommand{\ba}{{\bf a}}
\newcommand{\bb}{{\bf b}}
\newcommand{\bx}{{\bf x}}
\newcommand{\by}{{\bf y}}
\newcommand{\bk}{{\bf k}}
\newcommand{\bK}{{\bf K}}

\newcommand{\bR}{{\bf R}}
\newcommand{\re}{{\mathbb R}}
\newcommand{\C}{{\mathbb C}}

\def\beq{\begin{equation}}
\def\eeq{\end{equation}}
\newcommand{\bea}{\begin{eqnarray}}
\newcommand{\eea}{\end{eqnarray}}
\newcommand{\beas}{\begin{eqnarray*}}
\newcommand{\eeas}{\end{eqnarray*}}
\newcommand{\Pre}[1]{\ensuremath{\mathrm{Re} \left( #1 \right)}}
\newcommand{\Pim}[1]{\ensuremath{\mathrm{Im} \left( #1 \right)}}
{

\newcommand{\cS}{{\cal S}}
\newcommand{\cB}{{\cal B}}

\newcommand{\cF}{{\cal F}}
\newcommand{\cH}{{\cal H}}
\newcommand{\cK}{{\cal K}}
\newcommand{\cL}{{\cal L}}
\newcommand{\cU}{{\cal U}}

\begin{document}
\begin{center}
{\LARGE \bf On the guided states of \\ 3D biperiodic Schr\"odinger operators}

\medskip

\end{center}

\medskip

\begin{center}
{\sc F. Bentosela$^{\rm a, b}$, C. Bourrely$^{\rm c}$, Y. Dermenjian$^{\rm d}$, E. Soccorsi$^{\rm a, b}$}
\end{center}
$^{\rm a}${\em{Aix Marseille Universit\'e, CNRS, CPT, UMR 7332, 13288 Marseille, France}};
$^{\rm b}${\em{Universit\'e de Toulon, CNRS, CPT, UMR 7332, 83957 La Garde, France}};
$^{\rm c}${\em{Aix Marseille Universit\'e, D\'epartement de physique, Facult\'e des sciences de Luminy, 13288 Marseille, France}};
$^{\rm d}${\em{Aix Marseille Universit\'e, CNRS, LATP, UMR 7353, 13453 Marseille, France}}.

\begin{abstract}
We consider the Laplacian operator $H_0=-\Delta$ perturbed by a non-positive potential $V$, which is periodic in two directions, and decays in the remaining one. We are interested in the characterization and decay properties of the guided states, defined as the eigenfunctions of the reduced operators in the Bloch-Floquet-Gelfand transform of $H_0+V$ in the periodic variables. If $V$ is sufficiently small and decreases fast enough in the infinite direction, we prove that, generically, these guided states are characterized by quasi-momenta belonging to some one-dimensional compact real analytic submanifold of the Brillouin zone. Moreover they decay faster than any polynomial function in the infinite direction.\\



\end{abstract}


\section{Introduction}
\setcounter{equation}{0}
A better understanding of the mechanisms of radio wave propagation in complex environments is required for the development of efficient wireless communication systems. This can be achieved by investigating the corresponding Maxwell equations
where the source is a time periodic current density localized in the wires feeding the antenna,  $i(t,x) = {\rm e}^{- {\rm i} \omega t} i_0(x)$ , with frequency $\omega>0$. 
Of course these equations can not be solved explicitely in a rich scattering environment but they are the source of interesting mathematical problems. However there is only a very few number of theoretical results available describing radio waves in ``realistic" propagation media.
In the particular case of a periodic environment (for simplicity the medium is assumed to be periodic in, say, two out of the three directions, in order to model an idealized infinitely extended building) the problem is to determine the steady-state solutions $u(x,t) = {\rm e}^{- {\rm i} \omega t} u_0(x)$ associated to the solution $u_0$ to the equation $(H - k_0^2)u_0 = i_0$. Here $H$ denotes the selfadjoint Maxwell operator under study and $k_0= \omega \slash c$, where $c$ is the celerity of the propagation medium. If $i_0$ is nonzero and $k_0^2$ is a nonnegative number, it turns out that there is no $u_0$ satisfying the above equation in ${\rm L}^2(\re^3)$ in the general case. A potential outgoing solution can nevertheless be defined from the limit of the resolvent operator $(H - k_0^2 \mp  {\rm i} \varepsilon)^{-1} i_0$ as $\varepsilon \downarrow 0$. Similar problems have been studied in various contexts, in \cite{wilcox} for both acoustic and electromagnetic waves, in \cite{frank1},  \cite{frank2}, \cite{filonov1}-\cite{filErratum} and \cite{Ger} for the Schr\"odinger equation, and  in \cite{filonov2} for the Maxwell equations. 

The study carried out this article has some connection with \cite{bento1}, \cite{bento2} and \cite{bento3}. Actually, the model we had in mind when starting this work was inspired from \cite{bourrely}, where numerous numerical simulations of the scattering of electromagnetic waves in an infinitely extended and periodic building in two orthogonal directions are performed. However, in order to avoid numerous technical difficulties appearing in the treatment of the Maxwell operator, we consider in this paper the Schr\"odinger equation rather than the Maxwell equations.
 
The method used in \cite{frank1}, \cite{frank2}, \cite{filonov1}-\cite{filErratum}, \cite{filonov2}, \cite{Ger}
and \cite{wilcox} takes advantage of the periodicity of the system to decompose the corresponding partial differential operator $H$ into a direct integral $\int_{\cB} H(\bk_{\ell}) {\rm d} \bk_{\ell}$, where $\cB$ is the Brillouin zone and $\bk_{\ell}$ denotes the quasi-momentum associated to the Bloch-Floquet-Gelfand transform. In all the above mentioned papers, a Limiting Absorption Principle (LAP for short) for either $H$ or $H(\bk_{\ell})$, $\bk_{\ell} \in \cB$, is established in order to derive absolutely continuous properties of the spectrum. Convenient assumptions made on the periodic potential function $V$ describing the physical properties of the model under study then guarantee that the operator $\Gamma(\bk_{\ell},E \pm {\rm i} \varepsilon)=V(H(\bk_{\ell}) - E \mp {\rm i} \varepsilon)^{-1}$ is compact for all $\bk_{\ell} \in \cB$ and
$\varepsilon >0$. Here $E$ is $k_0^2$ in the case of the Maxwell equation while $E$ denotes the energy of the system for the Schr\"odinger equation. The next step
of the proof involves relating the LAP to the existence of  $\lim_{\varepsilon \downarrow 0}\int_{\cB}  \frac{f(E \pm {\rm i} \varepsilon,\bk_{\ell})}{h(E \pm {\rm i} \varepsilon,\bk_{\ell})} {\rm d} \bk_{\ell}$ for suitable functions $f$, where $h(E \pm {\rm i} \varepsilon,\bk_{\ell})$ denotes the Fredholm determinant of $\Gamma(\bk_{\ell},E \pm {\rm i} \varepsilon)$. This indicates that the derivation of a LAP for these models is intricately tied to the characterization of $\mathcal{C}(E)=\{ \bk_{\ell} \in \cB,\ h(\bk_{\ell},E)=0 \}$ (the so-called Fermi variety in \cite{Kuc}, by analogy to a similar -but slightly different- quantity in solid state physics)
which is generally a rather non trivial problem. The alternative to investigating the structure of $\mathcal{C}(E)$ in the process of establishing a LAP,  is to assume exponential decay for $V$ in the transversal (i.e. non periodic) direction, see \cite{filonov1}-\cite{filErratum}, \cite{filonov2}, \cite{Ger}. Furthermore the resolvent can be extended to a meromorphic function on the half lower or upper planes in this framework.
In this paper we develop a completely different approach based on the actual characterization of $\mathcal{C}(E)$, which turns out to be a real analytic manifold. The key advantage is that $\mathcal{C}(E)$ is thus parametrizable.
Based on standard computations, 
involving an appropriate change of variables in the above mentioned integrals when $\bk_{\ell}$ is in a neighborhood of $\mathcal{C}(E)$, this would allow us to derive a LAP for external potentials decreasing polynomially fast in the transversal direction. For the sake of brevity we shall nevertheless not mention the details of the related proof here. 

Another benefit of determining the underlying structure of $\mathcal{C}(E)$ is a better characterization of guided states occuring in this system, i.e. of the eigenfunctions of $H(\bk_{\ell})$ with energy $E$, for some $\bk_{\ell} \in \cB$. Indeed we prove in this framework that guided states correspond to quasi-momenta $\bk_{\ell}$ satisfying either $\bk_{\ell} \in \mathcal{C}(E)$ or $| \bk_{\ell} | = E^{1 \slash 2}$.
The terminology used here is justified by the fact that these eigenfunctions exhibit decay properties in the transversal direction. This will be made precise for both $\bk_{\ell} \in {\mathcal C}(E)$ and $| \bk_{\ell} | = E^{1 \slash 2}$, although the existence of guided states is only guaranteed for $\bk_{\ell} \in {\mathcal C}(E)$ in the general framework examined in this article. Nevertheless, the non existence of guided states associated to $| \bk_{\ell} | = E^{1 \slash 2}$ can be proven for a wide class of suitable periodic potential functions $V$ we shall fully describe. In this specific case, the guided states of the corresponding Hamiltonian are therefore characterized by quasi-momenta $\bk_{\ell} \in \mathcal{C}(E)$.
Notice that the terminology used in both mathematics and physics literature to classify waves depends quite strongly on the authors and the scientific communities they belong to. As a matter of fact the term ``surface state" is employed in \cite{bento1}, \cite{bento3} and \cite{wilcox}, while ``guided wave" is used in \cite{wilcox2}, and \cite{auld} refers to both ``surface" and ``guided" waves, depending upon the context.

\subsection{Settings and notations}
Let $\ba_2$, $\ba_3$ be two independent vectors in $\re^2$ generating a lattice
$\cL := \sum_{j=2,3} {\mathbb Z} \ba_j$.
We define the basic period cell (or Seitz zone) as
\bel{s1}
\cS := \re^2 \slash \cL = \left\{ \bx_{\ell} \in \re^2,\ \bx_{\ell}=\sum_{j=2,3} s_j \ba_j,\ -1 \slash 2 < s_j  \leq 1 \slash 2 \right\}.
\ee
Similarly, the dual basis $\bb_2$, $\bb_3$ being defined by
\bel{s2}
\langle \bb_i, \ba_j \rangle = (2 \pi) \delta_{i,j},\ i,j=2,3,
\ee
where $\langle . , . \rangle$ stands for the usual Euclidian scalar product in $\re^2$, the basic period cell for the dual basis $\{ \bb_1,\bb_2 \}$ (or Brillouin zone) is
\bel{s3}
\cB := \left\{ \bk_{\ell} \in \re^2,\ \bk_{\ell} = \sum_{j=2,3} t_j \bb_j,\ -1 \slash 2 < t_j  \leq 1 \slash 2 \right\}.
\ee
Evidently $\cB = \re^2 \slash \cL^{\perp}$ where $\cL^{\perp} := \sum_{j=2,3} {\mathbb Z} \bb_j$ is the reciprocal lattice.

For the sake of simplicity we assume in what follows that
$$
\langle \ba_2 , \ba_3 \rangle = 0,
$$
and hence
\bel{r12}
\bb_j = 2 \pi\frac{\ba_j}{| \ba_j |^2},\ j=2,3.
\ee

Let $V$ be a real-valued function on $\re^3$ obeying
\bel{s5}
V(x_1,\bx_{\ell}+\ba_j)=V(x_1,\bx_{\ell}),\ x_1 \in \re,\ \bx_{\ell}=(x_2,x_3) \in \re^2,\ j=2,3,
\ee
plus some technical additional conditions that will be mentioned further.
We are interested in the spectral properties of the self-adjoint realization of the operator $-\Delta+V$ acting in ${\rm L}^2(\re^3)$.

For all $\bk_{\ell} \in \cB$ we consider the operator $H_0(\bk_{\ell}) :=-\Delta$ in
$\cH := {\rm L}^2(\re \times \cS)$, with the boundary conditions
\bel{s6}
\varphi(x_1,\bx_{\ell}+\ba_j)={\rm e}^{{\rm i} \langle \bk_{\ell}, \ba_j \rangle} \varphi(x_1,\bx_{\ell}),\
\partial_j \varphi(x_1,\bx_{\ell}+\ba_j)={\rm e}^{{\rm i} \langle \bk_{\ell}, \ba_j \rangle} \partial_j \varphi(x_1,\bx_{\ell}),\ j=2,3,
\ee
for all $x_1 \in \re$ and $\bx_{\ell}$, with $\bx_{\ell}$ and $\bx_{\ell}+\ba_j$ in $\partial \cS$, the notation $\partial_j$ being understood as the derivative w.r.t. the coordinate $s_j$ such that $\bx_{\ell}=\sum_{j=2,3} s_j \ba_j$. 
Its domain is 
\bel{s6b}
{\rm dom}\ H_0(\bk_{\ell}) =\{ \varphi \in {\rm H}^2(\re \times \cS)\ {\rm fulfilling}\ \eqref{s6} \}
\ee
where ${\rm H}^{2}(\re \times \cS)$ denotes the usual second order Sobolev space.\\
Let us now consider the Bloch-Floquet-Gelfand transform $\cU : {\rm L}^2(\re^3) \rightarrow \int_{\cB}^{\oplus} \cH {\rm d} \bk_{\ell}$, defined on $C_0^{\infty}(\re^3)$ by
\bel{s4}
(\cU \varphi)(\bk_{\ell},x_1,\bx_{\ell}) := | \cB |^{-1 \slash 2} \sum_{\bR_{\ell} \in \cL} {\rm e}^{-{\rm i} \langle \bk_{\ell}, \bR_{\ell} \rangle} \varphi(x_1,\bx_{\ell}+\bR_{\ell})
\ee
and extended as a unitary operator on ${\rm L}^2(\re^3)$.



In light of \eqref{s4} and the identity
$\cU {\rm L}^2(\re^3) = \int_{\cB}^{\oplus} \cH {\rm d} \bk_{\ell}$, it turns out that
\bel{s7}
\cU ( -\Delta+V ) \cU^{-1} = \int_{\cB}^{\oplus} H(\bk_{\ell}) {\rm d} \bk_{\ell},
\ee
where $H(\bk_{\ell}):= H_0(\bk_{\ell})+V$ acts in $\cH$.

Let $\rho(H_0(\bk_{\ell}))$ be the resolvent set of $H_0(\bk_{\ell})$, and $\sigma(H_0(\bk_{\ell}))$ denote its spectrum. 
For all $\bk_{\ell} \in \cB$, let $R_0(\bk_{\ell},z):=(H_0(\bk_{\ell})-z)^{-1}$, $z \in \rho(H_0(\bk_{\ell})):= {\mathbb C} \backslash \sigma(H_0(\bk_{\ell}))$, and $R(\bk_{\ell},z):=(H(\bk_{\ell})-z)^{-1}$, $z \in \rho(H(\bk_{\ell}))$, be the resolvent operators associated to
$H_0(\bk_{\ell})$ and $H(\bk_{\ell})$ respectively. For all $z \in \rho(H_0(\bk_{\ell})) \cap \rho(H(\bk_{\ell}))$ we have
\bel{r1}
| V |^{1 \slash 2} R_0(\bk_{\ell},z) = (1 + \epsilon | V |^{1 \slash 2} R_0(\bk_{\ell},z) | V |^{1 \slash 2} ) | V |^{1 \slash 2} R(\bk_{\ell},z),
\ee
by combining the first resolvent equation
$R(\bk_{\ell},z) = R_0(\bk_{\ell},z) - R_0(\bk_{\ell},z) V R(\bk_{\ell},z)$,
with the obvious decomposition $V=\epsilon | V |^{1 \slash 2} | V |^{1 \slash 2}$, $\epsilon(\bx)$ being the sign of $V(\bx)$, $\bx:=(x_1,\bx_{\ell}) \in \re \times \cS$. Henceforth
\bel{r2}
R(\bk_{\ell},z) = R_0(\bk_{\ell},z) -  R_0(\bk_{\ell},z) \epsilon | V |^{1 \slash 2} (1 + \epsilon | V |^{1 \slash 2} R_0(\bk_{\ell},z) | V |^{1 \slash 2} )^{-1} | V |^{1 \slash 2} R_0(\bk_{\ell},z),
\ee
provided $1 + \epsilon | V |^{1 \slash 2} R_0(\bk_{\ell},z) | V |^{1 \slash 2}$ is boundedly invertible.

For the sake of simplicity we assume that $V(\bx):=-g W(\bx)^2$ for $\bx \in \re \times \cS$,
some coupling constant $g \in (0,+\infty)$ and some potential $W(\bx) \geq 0$.
Hence we have
\bel{r3}
R_g(\bk_{\ell},z) := R(\bk_{\ell},z) = R_0(\bk_{\ell},z) +  g R_0(\bk_{\ell},z) W (1 - g \Gamma_{\bk_{\ell}}(z) )^{-1} W R_0(\bk_{\ell},z),
\ee
from \eqref{r2}, where
\bel{r4}
\Gamma_{\bk_{\ell}}(z) := W R_0(\bk_{\ell},z) W,
\ee
is the integral operator in $\cH$ with kernel (see \cite{Ger})
\bel{r5}
\gamma_{\bk_{\ell}}(\bx,\by,z) := {\rm i} \beta W(\bx) W(\by) \sum_{\bK_{\ell} \in \cL^{\perp}}
\frac{{\rm e}^{{\rm i} \sqrt{z-|\bK_{\ell}+\bk_{\ell}|^2} |x_1 - y_1|}}{\sqrt{z-|\bK_{\ell}+\bk_{\ell}|^2}} {\rm e}^{{\rm i} \langle \bK_{\ell}+\bk_{\ell}, \bx_{\ell}-\by_{\ell} \rangle},
\ee
with
\bel{r5a}
\beta := \frac{| \cB |}{16 \pi^2},
\ee
the sign of the imaginary part of the square root of any complex number being chosen nonnegative. In \eqref{r5} and in
what follows, $| \bk_{\ell} |$ denotes the Euclidian norm of $\bk_{\ell} \in \re^2$, i.e. $| \bk_{\ell} |^2 = \sum_{j=2,3} k_j^2$ for $\bk_{\ell}=(k_2,k_3)$.

\subsection{Statement of the main results}
In this paper we essentially aim for examining the existence of guided states generated by a perturbation of the form $-g W^2$ for $g>0$. Borrowing the usual definition from \cite{wilcox2}, a guided state for the energy  $E$ 
is any nonzero function $u \in \cH$ satisfying the identity $H(\bk_{\ell}) u = E u$ for some $\bk_{\ell} \in \cB$. 
As follows from \eqref{r3}, the problem of the existence of guided states in this system is thus tied to the one of the invertibility of the family of operators $\{ 1- g \Gamma_{\bk_{\ell}}(E) \}_{\bk_{\ell} \in \cB}$ in $B(\mathrm{L}^2(\re \times \cS))$.

As a matter of fact, we shall consider in this framework energies $E \in (0,E_{\delta})$, $\delta \geq 0$, where 
\bel{r13}
E_{\delta} :=  \frac{\pi^2}{1+\delta} \min_{j=2,3} | \ba_j |^{-2} = \frac{1}{4 (1+\delta)} \min_{j=2,3} | \bb_j |^{2}.
\ee

The main result of this article is the following actual characterization of the subset 
$$\mathcal{C}_g(E) := \{ \bk_{\ell} \in \cB,\ \vert \bk_{\ell} \vert^{2}\not=E\ {\rm and}\ 1-g \Gamma_{\bk_{\ell}}(E)\ {\rm is\ singular} \},$$
for $g$ sufficiently small.
\begin{theorem}
\label{thm-a}
Let $\delta > 0$, $E \in (0,E_{\delta})$ , 
and let  $W \in {\rm L}^{\infty}(\re \times \cS)$ satisfy
\bel{rr0}
W(x_1,\bx_{\ell})  =O \left( \frac{1}{(1 + |x_1|)^{(3  + \epsilon) \slash 2}} \right) \mbox{ for some } \epsilon>0.
\ee
Then there exists $g_0=g_0(\delta,W)>0$ such that for all $g \in (0,g_0)$, the set $\mathcal{C}_g(E)$ is a one-dimensional compact real analytic manifold contained in an annulus $\Omega_g(E)$ centered at the origin, with smaller and bigger radius of size $E^{1/2}+ c_{\pm}g^2$, the two constants $c_{\pm}>0$ being independent of $g$. Namely, $\mathcal{C}_g(E)$ is a closed and simple (i.e. with no double point) curve which is not homotopic to a point in $\Omega_g(E)$.
\end{theorem}

Theorem \ref{thm-a} allows for a better characterization of guided states occuring in this system.
This can be seen from the coming result, which, first, establishes a link between $\mathcal{C}_g(E)$ and potential guided states, and, second, describes their rate of decay in the orthogonal direction to the periodic ``longitudinal" plane carrying $\bx_{\ell}$. Its statement 
actually involves the following family of ``weighted" subspaces of $\cH={\rm L}^2(\re \times \cS)$ :
$$\cH_{\tau} := \{ v \in \cH,\ (1+x_1^2)^{\tau \slash 2} v \in \cH \},\ \tau \geq 0.$$ 
Although these functional spaces are not the sharpest ones required for the following claim to hold, their use in this framework is justified by the fact that they allow for a more simple derivation of the result.

\begin{theorem}
\label{thm-sw}
Let $E$, $g_0$ and $W$ be the same as in Theorem \ref{thm-a}. 
Then for every $g \in (0,g_0)$, any guided state with energy $E$ belongs to $\cap_{m \in {\mathbb N}} \cH_{m}$, and is
associated to a quasi-momentum $\bk_{\ell}$ verifying either $\bk_{\ell} \in \mathcal{C}_g(E)$ or $| \bk_{\ell} | = E^{1 \slash 2}$.
Moreover,
for each $\bk_{\ell} \in \mathcal{C}_g(E)$, there exists at least one guided state with energy $E$ associated to $\bk_{\ell}$.
\end{theorem}
Although the existence of guided states associated to quasi-momenta $\bk_{\ell}$ satisfying $| \bk_{\ell} | = E^{1 \slash 2}$ cannot be ruled out in the general situation examined in Theorem \ref{thm-sw}, this is no longer the case 
when the potential function $W$ is sufficiently smooth w.r.t. $\bx_{\ell}$, and vanishes in 
a half-space parallel to the longitudinal direction. This and Theorem \ref{thm-sw} entails the following: 
\begin{follow}
\label{cor-sw}
Let $E$ and $W$ be as in Theorem \ref{thm-a}, and assume moreover that :
\begin{enumerate}[(i)]
\item $x_{\ell} \mapsto W(x_1,\bx_{\ell}) \in {\rm C}^4(\cS)$ for a.e. $x_1 \in \re$, and vanishes in a neighborhood of the boundary $\partial \cS$;
\item $W(x_1,\bx_{\ell})=0$ for a.e. $(x_1,\bx_{\ell}) \in I \times \cS$, where $I$ is any unbounded subinterval of $\re$.
\end{enumerate}
Then there exists $\tilde{g}_0>0$ such that for all $g \in (0,\tilde{g}_0)$, the quasi-momentum $\bk_{\ell}$ of any guided state with energy $E$, belongs to $\mathcal{C}_g(E)$.
\end{follow}
Actually, this result can be generalized, at the expense of greater technical difficulties, to a wider class of potentials $W$ than the one considered in this statement.\\
Notice that the generality of the results of Theorems \ref{thm-a}-\ref{thm-sw} and Corollary \ref{cor-sw} is not substantially restricted by the assumption $E \in (0,E_{\delta})$. Indeed this purely technical hypothesis allows for a clearer statement (see Remark \ref{rmk-ae}) and it can be checked from the following sections of this paper that these claims, subject to slight modifications, remain essentially true when dealing with $E$ outside $(0,E_{\delta})$.

\subsection{Outline}

The contents of this paper is as follows.
In Section \ref{sec-gen}, we collect basic properties of the Hilbert-Schmidt operators $\Gamma_{\bk_{\ell}}(E + {\rm i} \varepsilon)$, for $\bk_{\ell} \in \mathcal{B}\setminus\{\vert \bk_{\ell} \vert^2 =E \}$ and $\varepsilon \in \re$, needed in the proofs of the following sections, and explain the link between $\Omega_{g}(E)$ and the problem of the invertibility
of $1-g \Gamma_{\bk_{\ell}}(E)$. Section \ref{sec-spe} is devoted to the study of the invertibility of $1- g \Gamma_{\bk_{\ell}}(E)$ for sufficiently small $g>0$, and contains the proof of Theorem \ref{thm-a}. Section \ref{sec-sw} deals with the characterization of guided states occuring in this framework, and their decay properties in the direction orthogonal to $\bx_{\ell}$. It provides the proofs of Theorem
\ref{thm-sw} and Corollary \ref{cor-sw}.
The paper concludes with two appendices in Section \ref{sec-app}. The first appendix, Appendix A, contains the  proof of the Hilbert-Schmidt properties of the operators $\Gamma_{\bk_{\ell}}(E \pm {\rm i} \varepsilon)$, stated in Lemma \ref{lm-r1}. The second appendix, Appendix B, is dedicated to the derivation of a LAP for $H_0(\bk_{\ell},E \pm {\rm i} \varepsilon)$, $\varepsilon >0$ and $\bk_{\ell} \in \cB$, needed in the proof of Theorem \ref{thm-sw}. 


\section{Basic properties of $\Gamma_{\bk_{\ell}}(E+{\rm i} \varepsilon)$, $\bk_{\ell} \in \cB_E^{\pm}$, $\varepsilon \in \re$}
\setcounter{equation}{0}
\label{sec-gen}
We set 
\bel{r6a}
\cB_E^{\pm} := \{ \bk_{\ell} \in \cB,\ \pm ( | \bk_{\ell} |^2 - E ) >0 \},\ \cB_E := \cB_E^{-} \cup \cB_E^{+},
\ee
and we assume in the foregoing, without restricting the generality of this text, that 
\bel{r2b}
\| W \|_{\cH} = 1.
\ee

\subsection{Notations and summability result}
For all $\bk_{\ell}=\sum_{j=2,3} t_j \bb_j \in \cB$ and $\bK_{\ell}= \sum_{j=2,3} n_j \bb_j \in
\cL^{\perp}$, \eqref{s3}-\eqref{r12} entails
\bel{z1}
| \bk_{\ell} + \bK_{\ell} |^2 = \sum_{j=2,3} (n_j + t_j)^2 | \bb_j |^2,\ t_j \in (-1 \slash 2,1 \slash 2],\
n_j \in {\mathbb Z},\ j=2,3.
\ee
Thus for all $\delta \geq 0$ and $E \in (0,E_{\delta})$ we have
\bel{r6b}
E - | \bk_{\ell} + \bK_{\ell} |^2 < -\delta E_{\delta} ,\ \bk_{\ell} \in \cB,\ \bK_{\ell} \in \cL^{\perp} \backslash \{ 0\},
\ee
according to \eqref{r13}.
Further, in light of \eqref{r5} we define for all $\bk_{\ell} \in \cB_E^{\pm}$, $\bK_{\ell} \in \cL^{\perp}$ and $\varepsilon \in \re$,
\bea
p(\bk_{\ell}+\bK_{\ell},E+{\rm i} \varepsilon) & := & (E - | \bk_{\ell} + \bK_{\ell} |^2 + {\rm i} \varepsilon)^{1 \slash 2} \nonumber \\
& = & p_R(\bk_{\ell}+\bK_{\ell},E+{\rm i} \varepsilon) + {\rm i} p_I(\bk_{\ell}+\bK_{\ell},E+{\rm i} \varepsilon), \label{r6cg3}
\eea
where
\bel{r6cg4}
(p_R(\bk_{\ell}+\bK_{\ell},E+{\rm i} \varepsilon),  p_I(\bk_{\ell}+\bK_{\ell},E+{\rm i} \varepsilon) ) \in \re \times \re^+,
\ee
in such a way that
\bea
& & \gamma_{\bk_{\ell}}(\bx,\by,E+{\rm i} \varepsilon) \nonumber \\
& = & \beta W(\bx) W(\by) \sum_{\bK_{\ell} \in \cL^{\perp}}
\frac{{\rm e}^{-p_I(\bk_{\ell}+\bK_{\ell},E+{\rm i} \varepsilon) |x_1 - y_1|}}{- {\rm i} p(\bk_{\ell}+\bK_{\ell},E+{\rm i} \varepsilon)}{\rm e}^{{\rm i} p_R(\bk_{\ell}+\bK_{\ell},E+{\rm i} \varepsilon) |x_1 - y_1|} {\rm e}^{{\rm i} \langle \bk_{\ell}+\bK_{\ell}, \bx_{\ell}-\by_{\ell} \rangle}. \label{r6cg}
\eea
Actually $(p_R(\bk_{\ell}+\bK_{\ell},E+{\rm i} \varepsilon),p_I(\bk_{\ell}+\bK_{\ell},E+{\rm i} \varepsilon))$ is uniquely defined from
\eqref{r6cg3}-\eqref{r6cg4} for every $(\bk_{\ell},\bK_{\ell},\varepsilon) \notin \cB_E^- \times \{0 \} \times \{0 \}$. Indeed, since ${\rm arg}\ \sqrt{z}$ is taken in $[0,\pi)$ for any complex number $z=|z| {\rm e}^{{\rm i} {\rm arg}\ z}$ with ${\rm arg}\ z \in [0,2 \pi)$, we get, from \eqref{r6b}-\eqref{r6cg3},
\bel{rr1}
p_I(\bk_{\ell}+\bK_{\ell},E+{\rm i} \varepsilon)=\left( \frac{((E - | \bk_{\ell} + \bK_{\ell} |^2)^2 + \varepsilon^2)^{1 \slash 2}-(E - | \bk_{\ell} + \bK_{\ell} |^2)}{2} \right)^{1 \slash 2}, 
\ee
for every $(\bk_{\ell},\bK_{\ell},\varepsilon) \in \cB_E^{\pm} \times \cL^{\perp} \times \re$. This and \eqref{r6b} yields $p_I(\bk_{\ell}+\bK_{\ell},E+{\rm i} \varepsilon)>0$ for $(\bk_{\ell},\bK_{\ell},\varepsilon) \notin  \cB_E^- \times \{0 \} \times \{0 \}$, and hence
\bel{rr2}
p_R(\bk_{\ell}+\bK_{\ell},E+{\rm i} \varepsilon) = \frac{\varepsilon}{2 p_I(\bk_{\ell}+\bK_{\ell},E+{\rm i} \varepsilon)},\ (\bk_{\ell},\bK_{\ell},\varepsilon) \in (\cB_E^{\pm} \times \cL^{\perp} \times \re ) \backslash ( \cB_E^- \times \{0 \} \times \{0 \}).
\ee
However, if $(\bk_{\ell},\bK_{\ell},\varepsilon) \in \cB_E^{-} \times \{ 0 \} \times \{0 \}$, we have $p_I(\bk_{\ell},E)=0$ from \eqref{rr1}, which shows that $p_R(\bk_{\ell},E)$ cannot be defined by \eqref{rr2}. In this case we set
\bel{rr3}
p(\bk_{\ell},E) = p_R(\bk_{\ell},E) := (E - | \bk_{\ell} |^2)^{1 \slash 2} >0,\ \bk_{\ell} \in \cB_E^{-}.
\ee
Notice from \eqref{rr1}-\eqref{rr3} that we have:
\bel{rr3b}
\left\{ \begin{array}{lcl} \lim_{\varepsilon \downarrow 0} p(\bk_{\ell}+\bK_{\ell},E \pm {\rm i} \varepsilon) =p(\bk_{\ell}+\bK_{\ell},E) & {\rm if} &  (\bk_{\ell},\bK_{\ell}) \in (\cB_E^+ \times \cL^{\perp}) \cup (\cB_E^- \times \cL^{\perp} \backslash \{ 0 \}) \\
\lim_{\varepsilon \downarrow 0} p(\bk_{\ell},E \pm {\rm i} \varepsilon) = \pm p(\bk_{\ell},E) & {\rm if} & \bk_{\ell} \in \cB_E^-. \end{array} \right.
\ee
\begin{remark}
\label{rmk-ae}
The condition $E \in (0, E_{\delta})$  ensuring \eqref{r6b} for every $\bk_{\ell} \in \cB$  and $\bK_{\ell} \in \cL^{\perp} \backslash \{ 0 \}$, we have $p_I(\bk_{\ell}+\bK_{\ell},E)>0$ except for $\bK_{\ell}=0$ according to \eqref{rr1}.The picture is quite similar for $E \geq E_{\delta}$ in the sense that there exists only a finite set $\mathcal{K}_E$ of ``singular" values $\bK_ t \in \cL^{\perp}$ such that $E-|\bk_{\ell}+\bK_{\ell}|^2 \geq 0$ for some $\bk_{\ell} \in \cB$. This can be seen directly from \eqref{s3} and the identity $\cL^{\perp}=\sum_{j=2,3} {\mathbb Z} \bb_j$ through elementary computations. 
Therefore, the case $E \geq E_{\delta}$ can actually be handled in the same way as $E  \in (0, E_{\delta})$, at the expense of inessential greater technical difficulties, upon substituting $\mathcal{K}_E$ for $\{ 0 \}$  in the following computations\footnote{Such as the occurence of several connex components in $C_{g}(E)$.}. 
\end{remark}
For further reference we now establish the following
\begin{lemma}
\label{lm-r3}
Let $E \in (0,E_{\delta})$ for $\delta \geq 0$.
Then for every $\mu>2$ there exists a constant $\alpha_{\mu}(\delta) \in \re_+^*$, depending only on $\delta$ and $\mu$, such that we have
$$
\sum_{\bK_{\ell} \in \cL^{\perp} \backslash \{ 0 \}} \frac{1}{p_I(\bk_{\ell}+\bK_{\ell},E+{\rm i} \varepsilon)^{\mu}} \leq
\sum_{\bK_{\ell} \in \cL^{\perp} \backslash \{ 0 \}} \frac{1}{p_I(\bk_{\ell}+\bK_{\ell},E)^{\mu}} \leq \alpha_{\mu}(\delta), $$
for all $\bk_{\ell} \in \cB_E$, $E \in (0,E_{\delta})$ and $\varepsilon \in \re$.
\end{lemma}
\begin{proof}
For every $\bK_{\ell}=\sum_{j=2,3} n_j \bb_j \in \cL^{\perp} \backslash \{ 0 \}$ we get from \eqref{z1}-\eqref{r6b} and \eqref{rr1} that
$$
p_I(\bk_{\ell}+\bK_{\ell},E)^2 = - (E - | \bk_{\ell} + \bK_{\ell} |^2) \geq \sum_{j=2,3\ {\rm s.t.}\ n_j \neq 0} (|n_j| - 1 \slash 2)^2 | \bb_j |^2 - E .
$$
Henceforth we have
\bel{rr4}
p_I(\bk_{\ell}+\bK_{\ell},E) \geq \left( (1+\delta) \sum_{j=2,3\ {\rm s.t.}\ n_j \neq 0} ( 2 |n_j| - 1)^2 - 1  \right)^{1 \slash 2} E_{\delta}^{1 \slash 2},\ \bK_{\ell} \in \cL^{\perp} \backslash \{ 0 \},
\ee
by \eqref{r13}, since $E \in (0,E_{\delta})$. This, combined with the estimate
\bel{r20}
p_I(\bk_{\ell}+\bK_{\ell},E+{\rm i} \varepsilon) \geq p_I(\bk_{\ell}+\bK_{\ell},E),\ (\bk_{\ell},\bK_{\ell},\varepsilon) \in \cB_E^{\pm} \times \cL^{\perp} \times \re,
\ee
which immediately follows from \eqref{rr1} and \eqref{rr3}, proves the result.
\end{proof}

\subsection{Hilbert-Schmidt properties}
In view of collecting some Hilbert-Schmidt properties of the operators $\Gamma_{\bk_{\ell}}(E + {\rm i} \varepsilon)$, $\bk_{\ell} \in \cB_E$ and $\varepsilon \in \re$,  needed in the proofs of the following sections, let
\bel{r7}
P_{\bk_{\ell}}:= \langle . , \varphi_{\bk_{\ell}} \rangle_{\cH} \varphi_{\bk_{\ell}}, \bk_{\ell} \in \cB,
\ee
denote the projection operator onto the linear space spanned by the normalized function (we use \eqref{r2b})
\bel{r8}
 \varphi_{\bk_{\ell}}(\bx) := W(\bx) {\rm e}^{{\rm i} \langle \bk_{\ell} , \bx_{\ell} \rangle},\ \bx \in \re \times \cS.
\ee
From \eqref{r6cg3}-\eqref{r6cg} and \eqref{r7}-\eqref{r8} then follows for every $\bk_{\ell} \in \cB_E^{\pm}$ that
\bel{r8ag}
\Gamma_{\bk_{\ell}}(E+ {\rm i} \varepsilon) = \Lambda_{\bk_{\ell}}(E+{\rm i} \varepsilon) P_{\bk_{\ell}} + C_{\bk_{\ell}}(E+{\rm i} \varepsilon),
\ee
where
\bel{r8bg}
\Lambda_{\bk_{\ell}}(E+{\rm i} \varepsilon) := \frac{{\rm i} \beta}{p(\bk_{\ell},E+{\rm i} \varepsilon )},
\ee
and $C_{\bk_{\ell}}(E+{\rm i} \varepsilon)$ is the integral operator with kernel
\bel{r9g}
c_{\bk_{\ell}}(\bx,\by,E+{\rm i} \varepsilon)  :=  \gamma_{\bk_{\ell}}(\bx,\by,E+{\rm i} \varepsilon ) - \beta W(\bx) W(\by) \frac{{\rm e}^{{\rm i} \langle \bk_{\ell}, \bx_{\ell}-\by_{\ell} \rangle}}{-{\rm i} p(\bk_{\ell}, E+{\rm i} \varepsilon)}.
\ee

\begin{lemma}
\label{lm-r1}
Assume \eqref{rr0}. Then for every $E \in (0,E_{\delta})$, with $\delta \geq 0$, it holds true that:
\begin{enumerate}[(a)]
\item $C_{\bk_{\ell}}(E+{\rm i} \varepsilon)$ and $\Gamma_{\bk_{\ell}}(E+{\rm i} \varepsilon)$ are Hilbert-Schmidt operators for every $\bk_{\ell} \in \cB_E$ and $\varepsilon \in \re$;
\item $C_{\bk_{\ell}}(E+{\rm i} \varepsilon)^*=C_{\bk_{\ell}}(E-{\rm i} \varepsilon)$ and $\Gamma_{\bk_{\ell}}(E+{\rm i} \varepsilon)^*=\Gamma_{\bk_{\ell}}(E-{\rm i} \varepsilon)$ for all $(\bk_{\ell},\varepsilon) \in (\cB_E^- \times \re^*) \cup (\cB_E^{+} \times \re)$;
\item There exists a constant $c>0$ depending only on $\delta$ and $W$ satisfying
$$ \| C_{\bk_{\ell}}(E+{\rm i} \varepsilon) \|_{HS}  \leq c,\ (\bk_{\ell},\varepsilon) \in \cB_E \times \re, $$
where $\| . \|_{HS}$ stands for the Hilbert-Schmidt norm in $\re \times \cS$.
\item $\lim_{\varepsilon \downarrow 0} \Gamma_{\bk_{\ell}}(E \pm {\rm i} \varepsilon) = \Gamma_{\bk_{\ell}}(E)$ if $\bk_{\ell} \in \cB_E^+$ and $\lim_{\varepsilon \downarrow 0} \Gamma_{\bk_{\ell}}(E + {\rm i} \varepsilon) = \Gamma_{\bk_{\ell}}(E)$ if $\bk_{\ell} \in \cB_E^-$;
\item $\lim_{\varepsilon \downarrow 0} \Gamma_{\bk_{\ell}}(E - {\rm i} \varepsilon) = \Gamma_{\bk_{\ell}}(E) - 2 C_{\bk_{\ell}}^{(0)}(E)$ if $\bk_{\ell} \in \cB_E^-$, where $C_{\bk_{\ell}}^{(0)}(E)$ denotes the integral operator with kernel
    $$c_{\bk_{\ell}}^{(0)}(E,\bx,\by):= \beta W(\bx) W(\by)
\frac{{\rm e}^{{\rm i} p(\bk_{\ell},E) |x_1 - y_1|}}{-{\rm i} p(\bk_{\ell},E)} {\rm e}^{{\rm i} \langle \bk_{\ell}, \bx_{\ell}-\by_{\ell} \rangle},$$
\end{enumerate}
the limits in (d) and (e) being taken in the Hilbert-Schmidt norm sense.
\end{lemma}
The proof of this lemma being quite tedious, it is postponed to Appendix A in Section \ref{sec-applmr1}.

The operator $\Lambda_{\bk_{\ell}}(E +{\rm i} \varepsilon) P_{\bk_{\ell}}$ being normal from \eqref{r7}-\eqref{r8} and \eqref{r8bg}, with spectrum equal to $\{ 0,  \Lambda_{\bk_{\ell}}(E +{\rm i} \varepsilon) \}$, it
\footnote{We use the following result (see e.g. \cite{K}[Problem V-4.8]): Let $T$ be normal and $A \in \cB({\mathcal H})$, where ${\mathcal H}$ is an Hilbert space. Let $d(\zeta)={\rm dist}(\zeta,\sigma(T))$. Then $d(\zeta)>\|A\|$ implies $\zeta \in \rho(T+A)$ and $\| (T+A-\zeta)^{-1} \| \leq 1 \slash (d(\zeta)-\| A \|)$.}{follows} readily from \eqref{r8ag} and  Lemma \ref{lm-r1}(c) that:
\bel{ps1}
\sigma( \Gamma_{\bk_{\ell}}(E +{\rm i} \varepsilon) ) \subset \footnote{For all $z_0 \in {\mathbb C}$ and $r>0$, $\overline{B}(z_0,r)$ denotes the closure of the ball $B(z_0,r)=\{ z \in {\mathbb C},\ |z-z_0| < r \}$.}{\overline{B}(0,c)} \cup \overline{B}(\Lambda_{\bk_{\ell}}(E +{\rm i} \varepsilon),c),\ \bk_{\ell} \in \cB_E,\ E \in (0,E_{\delta}),\ \varepsilon \in \re.
\ee
Moreover $\Gamma_{\bk_{\ell}}(E)$ being selfadjoint for $\bk_{\ell} \in \cB_E^+$ by Lemma \ref{lm-r1}(b), then \eqref{ps1} yields:
\bel{ps2}
\sigma( \Gamma_{\bk_{\ell}}(E) ) \subset [-c,+c] \cup [\Lambda_{\bk_{\ell}}(E)-c,\Lambda_{\bk_{\ell}}(E)+c],\ \bk_{\ell} \in \cB_E^+,\ E \in (0,E_{\delta}).
\ee

\subsection{Spectral properties}
Since $\Gamma_{\bk_{\ell}}(E+{\rm i} \varepsilon)$ is compact for every $\bk_{\ell} \in \cB_E$ by Lemma \ref{lm-r1}(a), its spectrum $\sigma(\Gamma_{\bk_{\ell}}(E+{\rm i} \varepsilon))$ consists of at most a countable number of eigenvalues with finite multiplicity possibly excepting zero. 
In what follows we denote by $\lambda_1(\bk_{\ell},E+{\rm i} \varepsilon), \lambda_2(\bk_{\ell},E+{\rm i} \varepsilon), \ldots$ these nonzero eigenvalues arranged in non increasing order of magnitude
$$
| \lambda_1(\bk_{\ell},E+{\rm i} \varepsilon) | \geq | \lambda_2(\bk_{\ell},E+{\rm i} \varepsilon) | \geq \ldots,
$$
and call $P_1(\bk_{\ell},E+{\rm i} \varepsilon), P_2(\bk_{\ell},E+{\rm i} \varepsilon),\ldots$ the associated eigenprojections.

Fix $s>s_0$, where $s_0:=3+\sqrt{5}$. For all $g \in (0,s^{-1} c^{-1})$, we then introduce $\Omega_g(E)$, the layer invariant by rotation cited in Theorem \ref{thm-a}, that is defined by
\bel{r21dia}
\Omega_{g}(E) :=  \left\{ \bk_{\ell} \in \cB_E^+,\ p_I(\bk_{\ell},E) \in ( q_- g, q_+ g ) \right\},
\ee
with
\bel{r21star}
q_- =q_-(s):= \frac{s-1}{s} \beta\ {\rm and}\ q_+=q_+(s):= \frac{s-1}{s-2} \beta,
\ee
and then state the following

\begin{pr}
\label{pr-spe}
Let $E$ and $W$ be as in Theorem \ref{thm-a}
and let $g \in (0,s^{-1} c^{-1})$ where $s>s_0$ is fixed. Then
there exists $\varepsilon_0=\varepsilon_0(g)>0$ such that for all $\varepsilon \in (-\varepsilon_0,\varepsilon_0)$, we have:
\begin{enumerate}[(a)]
\item If $\bk_{\ell} \in \cB_E \backslash \Omega_g(E)$ then
$1 \in \rho(g \Gamma_{\bk_{\ell}}(E + {\rm i} \varepsilon))$ and $\| (1- g \Gamma_{\bk_{\ell}}(E + {\rm i} \varepsilon))^{-1} \| \leq 2 s (s-1)$;
\item  If $\bk_{\ell} \in \overline{\Omega}_g(E)$ then:
\begin{enumerate}[(i)]
\item $\lambda_1(\bk_{\ell},E + {\rm i} \varepsilon) \in \overline{B}(\Lambda_{\bk_{\ell}}(E + {\rm i} \varepsilon),c)$, and is a simple eigenvalue;
\item $\lambda_j(\bk_{\ell},E + {\rm i} \varepsilon) \in \overline{B}(0,c)$ for all $j \geq 2$;
\item $\overline{B}(0,c) \cap \overline{B}(\Lambda_{\bk_{\ell}}(E + {\rm i} \varepsilon),c) = \emptyset$.
\end{enumerate}
\end{enumerate}
Moreover it holds true that:
\begin{enumerate}[(c)]
\item $\mp g \lambda_1(\bk_{\ell},E) > \mp 1 + 1 \slash (s(s-1))$ if $\bk_{\ell} \in \cB_E^+$ is such that $\mp p_I(\bk_{\ell},E) < \mp q_{\pm} g$.
\end{enumerate}
\end{pr}
\begin{proof}
\noindent (a) Every $\bk_{\ell} \in \cB_E \backslash \Omega_g(E)$ satisfying one of the three following inequalities
$| \bk_{\ell} |^2 - E \geq q_+^2 g^2$, $| E - |\bk_{\ell}|^2 | \leq q_-^2 g^2$ or $E- |\bk_{\ell}|^2 > q_-^2 g^2$, we treat each corresponding case separately. If $| \bk_{\ell} |^2 - E > q_+^2 g^2$ it holds true that $|p(\bk_{\ell},E + {\rm i} \varepsilon)| = (( |\bk_{\ell}|^2 - E)^2+\varepsilon^2)^{1 \slash 4} \geq q_+ g$, whence
$g | \Lambda_{\bk_{\ell}}(E + {\rm i} \varepsilon) |  \leq (s-2) \slash (s-1)$ for every $\varepsilon \in \re$, by using \eqref{r8bg}. As a consequence we have
\bel{lap0a}
{\rm dist}(1,\sigma( g \Lambda_{\bk_{\ell}}(E + {\rm i} \varepsilon) P_{\bk_{\ell}}) ) \geq \frac{1}{s-1},\ \bk_{\ell} \in \cB_E^+\ {\rm s.t.}\ p_I(\bk_{\ell},E) \geq q_+ g,\ \varepsilon \in \re.
\ee
Similarly for $| E - | \bk_{\ell} |^2 | \leq q_-^2 g^2$ we get $|p(\bk_{\ell},E + {\rm i} \varepsilon)| \leq (q_-^4 g^4+\varepsilon^2)^{1 \slash 4}$ for each $\varepsilon \in \re$. Since $q_- < s \slash (s+1+\nu) \beta$ where $\nu:=1 \slash (2(s-1))$, it is true that $|p(\bk_{\ell},E + {\rm i} \varepsilon)| < s \slash (s+1+\nu) g \beta$, and hence that
$g | \Lambda_{\bk_{\ell}}(E + {\rm i} \varepsilon) |>1 + (1+\nu) \slash s$, provided $|\varepsilon | < \tilde{\varepsilon} := g^2 ( (s \slash (s+1+\nu))^4 \beta^4 - q_-^4 )^{1 \slash 2}$. As a consequence we have
\bel{lap0b}
{\rm dist}(1,\sigma( g \Lambda_{\bk_{\ell}}(E + {\rm i} \varepsilon) P_{\bk_{\ell}}) ) > \frac{1+\nu}{s},\ | E - |\bk_{\ell}|^2 | \leq q_-^2 g^2,\ \varepsilon \in (-\tilde{\varepsilon}, \tilde{\varepsilon}).
\ee
Last, in the case where $E- |\bk_{\ell}|^2 > q_-^2 g^2$, we see for every $\varepsilon \in \re$ that $| p(\bk_{\ell},E+{\rm i} \varepsilon) | > q_- g$ and
$| p_I(\bk_{\ell},E+{\rm i} \varepsilon) | \leq ( | \varepsilon | \slash 2)^{1 \slash 2}$, according to \eqref{rr1}. From this, \eqref{r8bg} and the identity $|\Pre{\Lambda_{\bk_{\ell}}(E + {\rm i} \varepsilon)}| = \beta |p_I(\bk_{\ell},E+{\rm i} \varepsilon)| \slash | p(\bk_{\ell},E+{\rm i} \varepsilon) |^2$ then follows that $g |\Pre{\Lambda_{\bk_{\ell}}(E + {\rm i} \varepsilon)}| < g \beta | \varepsilon |^{1 \slash 2} \slash (q_-^2 g^2) < (s - 1 - \nu) \slash s$ provided $|\varepsilon| < \hat{\varepsilon}:=g^2 \beta^{-2} q_-^4 ( (s -1 - \nu ) \slash s )^2$. This immediately entails
\bel{lap0c}
{\rm dist}(1,\sigma( g \Lambda_{\bk_{\ell}}(E + {\rm i} \varepsilon) P_{\bk_{\ell}}) )> \frac{1+\nu}{s},\ E- |\bk_{\ell}|^2 \geq q_-^2 g^2,\ \varepsilon \in (-\hat{\varepsilon}, \hat{\varepsilon}).
\ee
Now $g \Lambda_{\bk_{\ell}}(E + {\rm i} \varepsilon) P_{\bk_{\ell}}$ being normal, with $g \| C_{\bk_{\ell}}(E + {\rm i} \varepsilon) \|_{\mathcal{B}(\cH)} < 1 \slash s$ from Lemma \ref{lm-r1}(c), we may deduce from \eqref{lap0a}-\eqref{lap0c}
that $1 \in \rho(g \Gamma_{\bk_{\ell}}(E + {\rm i} \varepsilon))$ and $\| (1- g \Gamma_{\bk_{\ell}}(E + {\rm i} \varepsilon))^{-1} \|_{\mathcal{B}(\cH)} \leq 2 s (s-1)$ if $|\varepsilon| < \min(\tilde{\varepsilon}, \hat{\varepsilon})$.\\
\noindent (b) For all
$(\bk_{\ell},\varepsilon)  \in \overline{\Omega}_g(E) \times \re$ it holds true that
$| p(\bk_{\ell},E + {\rm i} \varepsilon) | = ( (| \bk_{\ell} |^2 - E )^2 + \varepsilon^2 )^{1 \slash 4}
\leq q_+ g  + | \varepsilon |^{1 \slash 2}$. This, combined with \eqref{r8bg}, \eqref{r21star}, and the fact that $g \in (0, 1 \slash (sc))$ with $s>s_0$, yields 
\bel{q3}
| \Lambda_{\bk_{\ell}}(E + {\rm i} \varepsilon) | > r:= \frac{s(s-2)}{s-1}c >  4 c,\
\bk_{\ell} \in \overline{\Omega}_g(E),\ | \varepsilon | < \check{\varepsilon} := q_+^2 \left( \frac{1}{sc} - g \right)^2,
\ee
by direct computations, and proves (iii).
Further, 
we fix $r_0 \in (2c, r \slash 2)$ and consider the circle ${\mathcal C}_{r_0}={\mathcal C}_{r_0}(\bk_{\ell},E,\varepsilon):=\{ \Lambda_{\bk_{\ell}}(E + {\rm i} \varepsilon) + r_0 {\rm e}^{{\rm i} \theta},\ \theta \in [0,2 \pi) \}$, so that we have
\bel{q6}
{\rm dist}({\mathcal C}_{r_0},\sigma(\Lambda_{\bk_{\ell}}(E + {\rm i} \varepsilon) P_{\bk_{\ell}})) \geq r_0,\ \bk_{\ell} \in \overline{\Omega}_g(E),\ \varepsilon \in (-\check{\varepsilon},\check{\varepsilon}).
\ee
As $r_0>c$, we deduce from \eqref{q6} and Lemma \ref{lm-r1}(c)
that ${\mathcal C}_{r_0} \subset \rho(\Gamma_{\bk_{\ell}}(E + {\rm i} \varepsilon))$, with
\bel{q6b}
 \| (z - \Gamma_{\bk_{\ell}}(E + {\rm i} \varepsilon))^{-1} \|_{\mathcal{B}(\cH)} \leq \frac{1}{r_0-c},\ z \in {\mathcal C}_{r_0},\
\bk_{\ell} \in \overline{\Omega}_g(E),\ \varepsilon \in (-\check{\varepsilon},\check{\varepsilon}).
\ee
This together with \eqref{q6} involves
\bea
& & \left\| \frac{1}{2 {\rm i} \pi} \int_{{\mathcal C}_{r_0}} ( \Gamma_{\bk_{\ell}}(E + {\rm i} \varepsilon) - z)^{-1} {\rm d} z - \frac{1}{2 {\rm i} \pi} \int_{{\mathcal C}_{r_0}} ( \Lambda_{\bk_{\ell}}(E + {\rm i} \varepsilon) P_{\bk_{\ell}} - z)^{-1} {\rm d} z \right\|_{\mathcal{B}({\rm L}^2(\re^3))} \nonumber \\
& \leq &  \frac{1}{2\pi} \int_{{\mathcal C}_{r_0}} \| (\Gamma_{\bk_{\ell}}(E + {\rm i} \varepsilon) - z)^{-1} C_{\bk_{\ell}}(E + {\rm i} \varepsilon) ( \Lambda_{\bk_{\ell}}(E + {\rm i} \varepsilon) P_{\bk_{\ell}} - z)^{-1} \|_{\mathcal{B}(\cH)}  {\rm d} s \nonumber \\
& \leq & \frac{c}{r_0-c} < 1,\ \bk_{\ell} \in \overline{\Omega}_g(E),\ \varepsilon \in (-\check{\varepsilon},\check{\varepsilon}). \label{q7}
\eea
Since $\Lambda_{\bk_{\ell}}(E + {\rm i} \varepsilon)$ is the only eigenvalue of $\Lambda_{\bk_{\ell}}(E + {\rm i} \varepsilon) P_{\bk_{\ell}}$ lying in the circle ${\mathcal C}_{r_0}$, and that it is non degenerate, \eqref{q7} then entails (see \cite{K}[Theorem I.6.32])
$$ \dim {\rm rank} \left( \frac{1}{2 {\rm i} \pi} \int_{{\mathcal C}_{r_0}} ( \Gamma_{\bk_{\ell}}(E + {\rm i} \varepsilon) - z)^{-1} {\rm d} z \right) = 1,\ \bk_{\ell} \in \overline{\Omega}_g(E),\ \varepsilon \in (-\check{\varepsilon},\check{\varepsilon}). $$
Therefore $\lambda_1(E + {\rm i} \varepsilon)$ is simple, and it is the only eigenvalue of $\Gamma_{\bk_{\ell}}(E + {\rm i} \varepsilon)$ lying in the disk $B(\Lambda_{\bk_{\ell}}(E + {\rm i} \varepsilon),r_0)$. Evidently (i) and (ii) follow from this, \eqref{ps1}, and the imbedding $\overline{B}(\Lambda_{\bk_{\ell}}(E + {\rm i} \varepsilon),c) \subset B(\Lambda_{\bk_{\ell}}(E + {\rm i} \varepsilon),r_0)$.\\
\noindent (c) For every $\bk_{\ell} \in \cB_E^+$ we have $p(\bk_{\ell},E)={\rm i} p_I(\bk_{\ell},E)$,
whence $\pm g \Lambda(\bk_{\ell},E) < \pm 1 + 1 \slash (s-1)$ from \eqref{r21star} and the assumption $\mp p_I(\bk_{\ell},E) > \mp q_{\pm} g$.
Bearing in mind Lemma \ref{lm-r1}(c), the result follows immediately from this, \eqref{r7}-\eqref{r8bg}, and the minimax principle, since $\Gamma_{\bk_{\ell}}(E)$ is selfadjoint in $\cH$ for every $\bk_{\ell} \in \cB_E^+$ by Lemma \ref{lm-r1}(b).
\end{proof}

\begin{remark}
Notice that Proposition \ref{pr-spe}(a),(c) actually hold true for every $s>2$;
\end{remark}

Bearing in mind that $gc <1$, Proposition \ref{pr-spe}(a),(b) immediately entails the
\begin{follow}
\label{cor-a1}
Under the assumptions of Proposition \ref{pr-spe}, the following equivalence holds true:
$$ \left( 1 \in \sigma(g \Gamma_{\bk_{\ell}}(E)),\ \bk_{\ell} \in \cB_E \right) \Longleftrightarrow \left( \bk_{\ell} \in \Omega_g(E)\ {\rm and}\ g \lambda_1(\bk_{\ell},E)= 1 \right).$$
\end{follow}

\subsection{Analyticity results}
Fix $(n_2,n_3) \in {\mathbb Z}^2$ and $\bK_{\ell}=\sum_{j=2,3} n_j \bb_j \in \cL^{\perp}$. For the purpose in hand we need to extend analytically the function $\bk_{\ell} \mapsto | \bk_{\ell} + \bK_{\ell} |^2$, defined in $\mathring{\cB}$ (interior of $\cB$), into $\mathring{\cB} + {\rm i} \re^2$. For every $(t_{j,R},t_{j,I}) \in (-1 \slash 2, 1 \slash 2) \times \re$, $j=2,3$, we thus define $\bk_{\ell,R}:= \sum_{j=2,3} t_{j,R} \bb_j$ and  $\bk_{\ell,I}:= \sum_{j=2,3} t_{j,I} \bb_j$, in such a way that
\bel{an0}
f_{\bK_{\ell}}(\bk_{\ell}) := ( \bk_{\ell} + \bK_{\ell} )^2 = \sum_{j=2,3} \left( (t_{j,R}+n_j)^2 - t_{j,I}^2  + 2 {\rm i}   (t_{j,R}+n_j) t_{j,I}  \right)  | \bb_j |^2,
\ee
is an analytic function in $\bk_{\ell} := \bk_{\ell,R} + {\rm i} \bk_{\ell,I} \in \mathring{\cB} + {\rm i} \re^2$. Therefore, $(\bk_{\ell},z) \mapsto z - f_{\bK_{\ell}}(\bk_{\ell})$ is analytic in
$(\mathring{\mathcal{B}} + {\rm i} \re^2) \times \C$.
Further, bearing in mind that $| \bk_{\ell,R} |^2 = \sum_{j=2,3} t_{j,R}^2 | \bb_j |^2 \leq | \bk_{\ell,R} + \bK_{\ell} |^2 = \sum_{j=2,3} (t_{j,R}+n_j)^2 | \bb_j |^2$ for every $\bk_{\ell,R} \in \mathring{\cB}$, and noticing that
$$ \Pre{z - f_{\bK_{\ell}}(\bk_{\ell})} = \Pre{z} - ( | \bk_{\ell,R} |^2 - | \bk_{\ell,I} |^2) = \Pre{z} - \sum_{j=2,3} ( (t_{j,R}+n_j)^2 - t_{j,I}^2 ) | \bb_j |^2, $$
we introduce the set
$$ \mathcal{D} := \{ (\bk_{\ell},z) \in (\mathring{\cB} + {\rm i} \re^2) \times \C,\ | \bk_{\ell,R} |^2 > \Pre{z} + | \bk_{\ell,I} |^2 \}. $$
Then, the mapping $\zeta \mapsto \sqrt{\zeta}$ being holomorphic on $\{ \zeta \in \C,\ \Pre{\zeta} < 0 \}$, $(\bk_{\ell},z) \mapsto p(\bk_{\ell}+\bK_{\ell},z) := \sqrt{z-f_{\bK_{\ell}}(\bk_{\ell})}$ is non vanishing and analytic in $\mathcal{D}$ for every $\bK_{\ell} \in \cL^{\perp}$. As a consequence $(\bk_{\ell},z) \mapsto \tilde{\gamma}_{\bk_{\ell}+\bK_{\ell}}(\bx,\by,z) := \frac{{\rm e}^{{\rm i} p(\bk_{\ell}+\bK_{\ell},z) |x_1-y_1|}}{-{\rm i} p(\bk_{\ell}+\bK_{\ell},z)} {\rm e}^{{\rm i} \langle \bk_{\ell} + \bK_{\ell}, \bx_{\ell} - \by_{\ell} \rangle}$ is thus holomorphic in $\mathcal{D}$ for every $\bK_{\ell} \in \cL^{\perp}$ and every $(\bx, \by) \in (\re \times \cS)^2$. Moreover, since
$| \tilde{\gamma}_{\bk_{\ell}+\bK_{\ell}}(\bx,\by,z)| \leq {\rm e}^{-p_I(\bk_{\ell}+\bK_{\ell},z) | x_1 - y_1 |} \slash p_I(\bk_{\ell}+\bK_{\ell},z)$ from \eqref{z1}-\eqref{r6b}, with, due to \eqref{rr1},
$$p_I(\bk_{\ell}+\bK_{\ell},z) := \Pim{p(\bk_{\ell}+\bK_{\ell},z)} \geq \left( \sum_{j=2,3} (|n_j| - 1 \slash 2)^2 | \bb_j |^2 - | \Pre{z} | \right)^{1 \slash 2}, $$
for each $\bK_{\ell}=\sum_{j=2,3} n_j \bb_j \in \cL^{\perp}$ such that $(n_2,n_3) \in ({\mathbb Z}^*)^2$, we deduce from \eqref{r4}-\eqref{r5a} that
$(\bk_{\ell},z) \mapsto \gamma_{\bk_{\ell}}(z,\bx,\by)$ is analytic in $\mathcal{D}$ for every $(\bx, \by) \in (\re \times \cS)^2$. This entails the

\begin{lemma}
\label{lm-an}
If two out of the three variables $(k_2,k_3,z) \in \mathcal{D}$ are fixed then $\Gamma_{\bk_{\ell}}(z)$ is a Kato analytic family of type (A) in the remaining variable.
\end{lemma}
In Lemma \ref{lm-an} we  used the obvious notation $\bk_{\ell} = (k_2,k_3)^T \in \C^2$.

\begin{lemma}
\label{lm-lan}
Let $E$, $g$ and $\varepsilon_0$ be as in Proposition \ref{pr-spe}. Then $(\bk_{\ell},\varepsilon) \mapsto \lambda_1(\bk_{\ell},E+{\rm i} \varepsilon)$ and $(\bk_{\ell},\varepsilon) \mapsto P_1(\bk_{\ell},E+{\rm i} \varepsilon)$ are real analytic in $\Omega_g(E) \times (-\varepsilon_0,\varepsilon_0)$, and continuous in $\overline{\Omega}_g(E) \times [-\varepsilon_0,\varepsilon_0]$.
\end{lemma}
\begin{proof}
For all $\epsilon \in (0,\min(E,E_{\delta}-E,q_-^2 g^2))$ we consider the set $\Omega_g(E,\epsilon)$ (resp. $\overline{\Omega}_g(E,\epsilon)$) of pairs $(\bk_{\ell},z)$ in the Cartesian product of $\mathring{\cB} + {\rm i} \re^2$ and $(E-\epsilon,E+\epsilon) + {\rm i} (-\varepsilon_0,\varepsilon_0)$, satisfying the condition $q_-^2 g^2  < | \bk_{\ell,R} |^2 - | \bk_{\ell,I} |^2 -E < q_+^2 g^2$ (resp. $q_-^2 g^2  \leq | \bk_{\ell,R} |^2 - | \bk_{\ell,I} |^2 -E \leq q_+^2 g^2$). It is clear that $\overline{\Omega}_g(E,\epsilon) \subset \mathcal{D}$ since $| \bk_{\ell,R} |^2 - | \bk_{\ell,I} |^2 \geq E + q_-^2 g^2> E + \epsilon >  \Pre{z} $ for each $(\bk_{\ell},z) \in \overline{\Omega}_g(E,\epsilon)$. Further $\lambda_1(\bk_{\ell},z)$ being simple for every $(\bk_{\ell},z) \in \overline{\Omega}_g(E,\epsilon)$ by Proposition \ref{pr-spe}(b)(i), we may deduce from Lemma \ref{lm-an} that
$\lambda_1$ is analytic in each variable $k_2$, $k_3$ or $z$ separately, when the two other complex variables are fixed. Separate analyticity implies joint analyticity by Hartogs' theorem (see e.g. \cite{H}[Theorem 2.2.8]) and
$(\bk_{\ell},z) \mapsto \lambda_1(\bk_{\ell},z)$ is thus analytic in $\Omega_g(E,\epsilon)$. Similarly, the continuity of $(\bk_{\ell},z) \mapsto \lambda_1(\bk_{\ell},z)$ in $\overline{\Omega}_g(E,\epsilon)$ follows from Proposition \ref{pr-spe}(b)(i) and \cite{K}[Chap. II-\S 5.7].
The case of $P_1$ is treated in a similar way. Indeed, arguing as before, we see that $P_1$ is an analytic function in each variable $k_2$, $k_3$ or $z$ separately, when the two other variables are fixed, whence $(\bk_{\ell},z) \mapsto P_1(\bk_{\ell},z)$ is weakly analytic on $\Omega_g(E,\epsilon)$ by Hartogs' theorem. Since any weakly analytic vector-valued function
on $\Omega_g(E,\epsilon)$ is extendable to an analytic function in the usual sense in a neighborhood of $\Omega_g(E,\epsilon)$ according to \cite{Vas}[Proposition 7.6], the proof is now complete.
\end{proof}

\section{On the invertibility of $1-g \Gamma_{\bk_{\ell}}(E)$, $\bk_{\ell} \in \cB_E$}
\setcounter{equation}{0}
\label{sec-spe}

\subsection{The equation $g \lambda_1(\bk_{\ell},E)-1=0$: proof of Theorem \ref{thm-a}}
In light of Corollary \ref{cor-a1},  we are left with the task of studying the set
\bel{r6}
\mathcal{C}_g(E) = \{ \bk_{\ell} \in \Omega_g(E),\ g \lambda_1(\bk_{\ell},E)-1=0 \}.
\ee
To this purpose we start by describing the behavior of $\lambda_1(\bk_{\ell},E)$ w.r.t. $\bk_{\ell}$.

\begin{pr}
\label{pr-a2}
Let $E$, $W$, $s$ and $g$ be as in Proposition \ref{pr-spe}.
Then the mappings $\bk_{\ell} \mapsto \lambda_1(\bk_{\ell},E)$ and $\bk_{\ell} \mapsto P_1(\bk_{\ell},E)$ are both continuous in $\overline{\Omega}_g(E)$, and real analytic in $\Omega_g(E)$. Moreover there are two constants $g_0=g_0(\delta,s,W)>0$ and $\varkappa=\varkappa(\delta,s,W)>0$, such that we have
$$ \sum_{j=2,3} \left| \frac{\partial \lambda_1}{\partial k_j}(\bk_{\ell},E) \right| \geq \varkappa >0,\ \bk_{\ell} \in \Omega_g(E),\ E \in (0,E_{\delta}),\ g \in (0,g_0).  $$
\end{pr}
The proof of Proposition \ref{pr-a2} being rather technical and lengthy, it is postponed to \S \ref{subsec-proofpr-a2} below. Armed with Proposition \ref{pr-a2} we may now complete the proof of Theorem \ref{thm-a}.

\paragraph{Proof of Theorem \ref{thm-a}.}
Let us first show that $\mathcal{C}_g(E) \neq \emptyset$. To see this we consider any continuous curve $\mathcal{C}$ in $\cB_E^+$ joining two points $\bk_{\ell}^{\pm}$ belonging to each of the two connex components of the boundary $\partial \Omega_g(E)$. Since $\bk_{\ell} \mapsto \lambda_1(\bk_{\ell},E)$ is continuous in $\mathcal{C}$ from
the first part of Proposition \ref{pr-a2}, then the intermediate value theorem combined with Proposition \ref{pr-spe}(c) entails that there is at least one $\tilde{\bk}_t \in \mathcal{C} \cap \Omega_g(E)$ satisfying $\lambda_1(\tilde{\bk}_t,E)=1$. Hence $\tilde{\bk}_t \in \mathcal{C}_g(E)$.

Next, $\bk_{\ell} \mapsto \lambda_1(\bk_{\ell},E)$ being a real analytic submersion from $\Omega_g(E)$ into $\re$, according to Proposition \ref{pr-a2}, the closed set $\mathcal{C}_g(E)$ is thus necessarily a real analytic submanifold with no endpoint. In other words, $\mathcal{C}_g(E)$ is a closed, regular, analytic curve in $\Omega_g(E)$.

Further, if $\mathcal{C}_g(E)$ were homotopic to a point in $\Omega_g(E)$, then the mapping $\bk_{\ell} \mapsto \lambda_1(\bk_{\ell},E)$ would admit an extremum in the compact set $\Delta_g(E)$ enclosed by $\mathcal{C}_g(E)$, and its gradient would thus vanish in at least one point of $\Delta_g(E)$. This claim being in contradiction to Proposition \ref{pr-a2}, $\mathcal{C}_g(E)$ is thus not homotopic to a point in $\Omega_g(E)$. 
This completes the proof.\\

\subsection{Proof of Proposition \ref{pr-a2}}
\label{subsec-proofpr-a2}
Since $\bk_{\ell} \mapsto \lambda_1(\bk_{\ell},E)$ is  analytic in $\Omega_g(E)$ by Lemma \ref{lm-lan}, we have
$$
\frac{\partial \lambda_1}{\partial k_j}(\bk_{\ell},E) = \int_{( \re \times \cS )^2} \frac{\partial \gamma_{\bk_{\ell}}}{\partial k_j}(\bx,\by,E) \psi_1(\bx) \psi_1(\by) {\rm d} \bx {\rm d} \by,\ \bk_{\ell} \in \Omega_g(E),\ j=2,3,
$$
from the Feynman-Hellmann formula, where $\psi_1 = \psi_1(\bk_{\ell}, E)$ denotes a $\cH$-normalized and real valued eigenfunction of $\Gamma_{\bk_{\ell}}(E)$ associated to $\lambda_1(\bk_{\ell},E)$.
From this and \eqref{r6b}-\eqref{rr1} then follows that
\bea
\frac{\partial \lambda_1}{\partial k_j}(\bk_{\ell},E)
& = & -\int_{( \re \times \cS )^2} \beta W(\bx) W(\by) \sum_{\bK_{\ell} \in \cL^{\perp}}
\frac{k_j + K_j}{p(\bk_{\ell}+\bK_{\ell},E)^3}(p(\bk_{\ell}+\bK_{\ell},E) |x_1-y_1| + 1)  \nonumber \\
& & \hspace{1.5cm} \times {\rm e}^{-p(\bk_{\ell}+\bK_{\ell},E) |x_1 - y_1|} {\rm e}^{{\rm i} \langle \bk_{\ell}+\bK_{\ell}, \bx_{\ell}-\by_{\ell} \rangle} \psi_1(\bx) \psi_1(\by) {\rm d} \bx {\rm d} \by, \label{n5}
\eea
for all $\bk_{\ell} \in \Omega_g(E)$ and $j=2,3$, where we used the identity
\beas
& & \sum_{\bK_{\ell} \in \cL^{\perp}} \int_{(\re \times \cS)^2} W(\bx) W(\by) \frac{{\rm e}^{-p(\bk_{\ell}+\bK_{\ell},E) |x_1 - y_1|}}{p(\bk_{\ell}+\bK_{\ell},E)} {\rm e}^{{\rm i} \langle \bk_{\ell} + \bK_{\ell} , \bx_{\ell} - \by_{\ell} \rangle}  (x_j-y_j) \psi_1(\bx) \psi_1(\by)
{\rm d} \bx {\rm d} \by  \\
& = & \langle \Gamma_{\bk_{\ell}}(E)  \psi_1, x_j \psi_1 \rangle_{\cH} - \langle y_j \psi_1, \Gamma_{\bk_{\ell}}(E) \psi_1 \rangle_{\cH}=0.
\eeas
The strategy consists of splitting the sum in the r.h.s. of \eqref{n5} into two parts, $S_0$ and $S_*$, corresponding respectively to $\bK_{\ell}=0$ and $\bK_{\ell} \neq 0$. We shall treat $S_0$ and $S_*$ separately.

Let us start dealing with $S_*$. First, by performing the change of variable $u=y_1-x_1$ in the following integral, we notice for $m=0,1$, and a.e. $y_1 \in \re$, that
$$
\int_{\re} |x_1 - y_1|^m {\rm e}^{-p(\bk_{\ell}+\bK_{\ell},E) |x_1 - y_1|} W(x_1,\bx_{\ell}) \psi_1(x_1,\bx_{\ell}) {\rm d} x_1
= \left( f_m \star W(.,\bx_{\ell}) \psi_1(.,\bx_{\ell}) \right)(y_1)
$$
where $f_m(u):= |u|^m {\rm e}^{-p(\bk_{\ell}+\bK_{\ell},E) |u|}$, $u \in \re$. Since $f_m \in {\rm L}^1(\re)$
with $\| f_m \|_{{\rm L}^1(\re)}=2 p(\bk_{\ell}+\bK_{\ell},E)^{-m+1}$, this yields
\beas
 \left\| \int_{\re} |x_1 - y_1|^m {\rm e}^{-p(\bk_{\ell}+\bK_{\ell},E) |x_1 - y_1|} W(\bx) \psi_1(\bx) {\rm d} x_1 \right\|_{{\rm L}^2(\re)} 
\leq \frac{2 \| W(.,\bx_{\ell}) \|_{{\rm L}^{\infty}(\re)} \| \psi_1(.,\bx_{\ell}) \|_{{\rm L}^2(\re)}}{p(\bk_{\ell}+\bK_{\ell},E)^{m+1}},
\eeas
and hence
\beas
& & \left| \int_{\re^2} |x_1 - y_1|^m {\rm e}^{-p(\bk_{\ell}+\bK_{\ell},E) |x_1 - y_1|} W(\bx) W(\by) \psi_1(\bx) \psi_1(\by) {\rm d} x_1 {\rm d} y_1 \right| \\
& \leq & \frac{2}{p(\bk_{\ell}+\bK_{\ell},E)^{m+1}} \| W(.,\bx_{\ell}) \|_{{\rm L}^{\infty}(\re)}
\| W(.,\by_{\ell}) \|_{{\rm L}^{\infty}(\re)} \| \psi_1(.,\bx_{\ell}) \|_{{\rm L}^2(\re)} \| \psi_1(.,\by_{\ell}) \|_{{\rm L}^2(\re)}.
\eeas
Thus, by integrating the above inequality w.r.t. $\bx_{\ell}$ and $\by_{\ell}$ over $\cS$ and using the normalization condition $\| \psi_1 \|_{\cH}=1$, we find out that
\bea
& & \left| \int_{(\re \times \cS)^2} |x_1 - y_1|^m {\rm e}^{-p(\bk_{\ell}+\bK_{\ell},E) |x_1 - y_1|} W(\bx) W(\by) \psi_1(\bx) \psi_1(\by) {\rm e}^{{\rm i} \langle \bk_{\ell}+\bK_{\ell}, \bx_{\ell}-\by_{\ell} \rangle} {\rm d} \bx {\rm d} \by \right| \nonumber \\
& \leq & \frac{2}{p(\bk_{\ell}+\bK_{\ell},E)^{m+1}} \| W \|_{{\rm L}^{\infty}(\re,{\rm L}^2(\cS))}^2.
\label{n8}
\eea
Further, since $|k_j+K_j| \leq (p(\bk_{\ell}+\bK_{\ell},E)^2 +E)^{1 \slash 2} \leq  p(\bk_{\ell}+\bK_{\ell},E)+ E_{\delta}^{1 \slash 2}$ for all $\bK_{\ell} \in \cL^{\perp}$, $E \in (0,E_{\delta})$ and $j=2,3$, it follows from \eqref{n8} and Lemma \ref{lm-r3} that
\bel{n9}
| S_*| \leq 4 \beta ( \alpha_3(\delta)+ \alpha_4(\delta) E_{\delta}^{1 \slash 2}) \| W \|_{{\rm L}^{\infty}(\re,{\rm L}^2(\cS))}^2 :=\sigma_*(\delta,W).
\ee
We turn now to estimating $S_0$, which is brought into the form
\bea
S_0 & = & \beta \frac{k_j}{p(\bk_{\ell},E)^3} \left( \int_{( \re \times \cS )^2} W(\bx) W(\by) \kappa(x_1,y_1) {\rm e}^{{\rm i} \langle \bk_{\ell}, \bx_{\ell}-\by_{\ell} \rangle} \psi_1(\bx) \psi_1(\by) {\rm d} \bx {\rm d} \by  \right. \nonumber \\
& & \hspace{2.0cm} \left. -\int_{( \re \times \cS )^2} W(\bx) W(\by) {\rm e}^{{\rm i} \langle \bk_{\ell}, \bx_{\ell}-\by_{\ell} \rangle} \psi_1(\bx) \psi_1(\by) {\rm d} \bx {\rm d} \by  \right), \label{n10}
\eea
where
\bel{n10b}
\kappa(x_1,y_1) :=  1 - {\rm e}^{-p(\bk_{\ell},E) |x_1 - y_1|} - p(\bk_{\ell},E) |x_1-y_1|  {\rm e}^{-p(\bk_{\ell},E) |x_1 - y_1|}.
\ee
Bearing in mind \eqref{r2b} and \eqref{r8}, we may express the second term in the r.h.s. of \eqref{n10} as
\bel{n11a}
\int_{( \re \times \cS )^2} W(\bx) W(\by) {\rm e}^{{\rm i} \langle \bk_{\ell}, \bx_{\ell}-\by_{\ell} \rangle} \psi_1(\bx) \psi_1(\by) {\rm d} \bx {\rm d} \by = | \langle \varphi_{\bk_{\ell}} , \psi_1 \rangle |^2.
\eeq
The benefits of \eqref{n11a} is the following lower estimate:
\bel{r23}
| \langle \varphi_{\bk_{\ell}}, \psi_1 \rangle | \geq \left( 1 - \frac{4(s-1)}{s(s-2)} \right)^{1 \slash 2},\ \bk_{\ell} \in \Omega_g(E).
\ee
Indeed, we know from \eqref{q7} that the spectral
projection $P_1=\langle .,\psi_1 \rangle \psi_1$ satisfies
$\| P_1 - P_{\bk_{\ell}} \|_{\mathcal{B}(\cH)}   \leq c \slash (r_0-c)$ for any $\bk_{\ell} \in \Omega_g(E)$ and
$r_0 \in (2c , r \slash 2)$, with $r=s(s-2) \slash (s-1) c$.
Taking $r_0 = c + r \slash 4$ we obtain
\bel{r23b}
\| P_1 - P_{\bk_{\ell}} \| \leq \frac{4(s-1)}{s(s-2)},\ \bk_{\ell} \in \Omega_g(E).
\ee
Moreover, since
$| \langle \varphi_{\bk_{\ell}} , \psi_1 \rangle |^2  =  \langle \psi_1 , \varphi_{\bk_{\ell}} \rangle \langle \varphi_{\bk_{\ell}} , \psi_1 \rangle=\langle \psi_1, P_{\bk_{\ell}} \psi_1\rangle=1 + \langle \psi_1 , (P_{\bk_{\ell}}-P_1) \psi_1 \rangle$,
we get that
$ | \langle \varphi_{\bk_{\ell}} , \psi_1 \rangle |^2 \geq 1 - \| P_{\bk_{\ell}} - P_1 \|_{\mathcal{B}(\cH)} $,
which, together with \eqref{r23b}, entails \eqref{r23}.

Let us now study the first term in the r.h.s. of \eqref{n10}. Taking into account that $1-{\rm e}^{-u} \in [0,u]$ for every $u \geq 0$, we get that $| \kappa(x_1,y_1) | \leq p(\bk_{\ell},E) (|x_1|+|y_1|)$ for all $(x_1,y_1) \in \re^2$, and hence
\bel{n11}
\left|  \int_{( \re \times \cS )^2} W(\bx) W(\by) \kappa(x_1,y_1) {\rm e}^{{\rm i} \langle \bk_{\ell}, \bx_{\ell}-\by_{\ell} \rangle} \psi_1(\bx) \psi_1(\by) {\rm d} \bx {\rm d} \by \right|
\leq 2 p(\bk_{\ell},E) \| W \|_{\cH_1},
\ee
according to \eqref{r2b}. Finally, putting \eqref{n5}, \eqref{n9}-\eqref{n11a} and \eqref{n11} together, we find out that
$$
\left| \frac{\partial \lambda_1}{\partial k_j}(\bk_{\ell},E) + \beta \frac{k_j}{p(\bk_{\ell},E)^3} | \langle \varphi , \psi_1 \rangle |^2 \right|
\leq 2 \beta \frac{|k_j|}{p(\bk_{\ell},E)^{2}} \| W \|_{\cH_1} + \sigma_*(\delta,W), $$
for all $\bk_{\ell} \in \Omega_g(E)$.
This, combined with \eqref{r21dia}-\eqref{r21star} and \eqref{r23}, yields
$$\left|  \frac{\partial \lambda_1}{\partial k_j}(\bk_{\ell},E) \right|
\geq \frac{ |k_j|}{\beta g^{3}} \sigma_0(s,W,g) - \sigma_*(\delta,W),\ \bk_{\ell} \in \Omega_g(E),
$$
where
\bel{n13}
\sigma_0(s,W,g) := \beta^{-1} \left( \frac{s-2}{s-1} \right)^3 \left( 1 - \frac{4(s-1)}{s(s-2)} \right) -2 \left( \frac{s}{s-1} \right)^2 \| W \|_{\cH_1} g.
\ee
As a consequence we have
$\sum_{j=2}^3 \left|  \frac{\partial \lambda_1}{\partial k_j}(\bk_{\ell},E) \right|
\geq \beta^{-1} |\bk_{\ell}| g^{-3} \sigma_0(s,W,g) - 2 \sigma_*(\delta,W).
$
So, for all $\bk_{\ell} \in \Omega_g(E)$, we have 
$\sum_{j=2}^3 \left|  \frac{\partial \lambda_1}{\partial k_j}(\bk_{\ell},E) \right|
\geq  \frac{s-1}{s} g^{-2} \sigma_0(s,W,g) - 2 \sigma_*(\delta,W)$
since $| \bk_{\ell} | \geq q_-g \geq ((s-1) \slash s) \beta g$ according to \eqref{r21dia}-\eqref{r21star}.
This and \eqref{n13} entails Proposition \ref{pr-a2}.

\section{Characterization of the guided states: proofs of Theorem \ref{thm-sw}} 
\label{sec-sw}

We proceed with a succession of six lemmas, the first five listed leading to the conclusion of Theorem \ref{thm-sw}, while the last one involves Corollary \ref{cor-sw}. 

\subsection{Quasi-momentum $\bk_{\ell}$ associated to guided states}

We first collect the following preliminary result that will be needed in the proof of Lemma \ref{lm-sw2}.

\begin{lemma}
\label{lm-sw1}
Let $E$, $g$ and $W$ be as in Theorem \ref{thm-a}. Then any guided state $u$ with energy $E$, associated to $\bk_{\ell} \in \cB_E$, verifies
$u = g R_0(\bk_{\ell},E \pm {\rm i} 0) W^2 u$ and, consequently, $Wu\not=0$.
\end{lemma}
\begin{proof}
The guided state 
$u$ being solution to the equation
$(H_0(\bk_{\ell})-(E \pm  {\rm i} \varepsilon)) u = (g W^2 \mp {\rm i} \varepsilon) u$
for every $\varepsilon \in \re_+$, we have
\bel{sw1}
u = g  R_0(\bk_{\ell},E \pm {\rm i} \varepsilon) W^2 u \mp {\rm i} \varepsilon R_0(\bk_{\ell},E \pm {\rm i} \varepsilon) u,\ \varepsilon >0.
\ee
If $\bk_{\ell} \in \cB_E^+$ we obtain the desired result directly from Proposition \ref{pr-lapz} (a) and (b), by taking the limit in \eqref{sw1} as $\varepsilon \downarrow 0$. Doing the same for $\bk_{\ell} \in \cB_E^-$ after
noticing that $W^2 u \in \cH_{3 + \epsilon}$,
we find that $u = g R_0(E \pm {\rm i} 0) u$ in $\cH_{-3-\epsilon}$ weakly.
Further, $u$ being in $\cH$, the equality holds true in $\cH$ and the proof is complete.
\end{proof}

Notice from Proposition \ref{pr-lapz} that Lemma \ref{lm-sw1} simply reads $u=g R_0(\bk_{\ell},E) u$ for $\bk_{\ell} \in \cB_E^+$.

Armed with Lemma \ref{lm-sw1} we may now characterize the quasi-momenta corresponding to actual guided states.

\begin{lemma}
\label{lm-sw2}
Let $E$, $g$ and $W$ be the same as in Theorem \ref{thm-a}, and assume that there exists a guided state with energy $E$ associated to $\bk_{\ell} \in \cB$.
Then we have either $\bk_{\ell} \in \mathcal{C}_g(E)$ or $| \bk_{\ell} | =  E^{1 \slash 2}$.
\end{lemma}
\begin{proof}
Let $u$ denote a guided state associated to $\bk_{\ell} \in \cB_E$. Then $v=Wu$ is solution to the equation 
\bel{sw1b}
(1 - g \Gamma_{\bk_{\ell}}(E \pm {\rm i} 0)) v = 0,
\ee
according to \eqref{r4} and Lemma \ref{lm-sw1}. From this and the first part of Lemma \ref{lm-r1}(d) then follows that $(1 - g \Gamma_{\bk_{\ell}}(E)) v = 0$ if $\bk_{\ell} \in \cB_E^+$. Similarly, putting \eqref{sw1b} and the second claim of Lemma \ref{lm-r1} (d)  together with Lemma \ref{lm-r1} (e) for $\bk_{\ell} \in \cB_E^-$, we see that $C_{\bk_{\ell}}^{(0)} v = 0$ as $\Gamma_{\bk_{\ell}}(E+{\rm i} 0) v = \Gamma_{\bk_{\ell}}(E - {\rm i} 0) v$, and hence that $(1 - g \Gamma_{\bk_{\ell}}(E)) v = 0$. Since $v \in \cH \backslash \{ 0 \}$ by Lemma \ref{lm-sw1}, we thus have $1 \in \sigma(g \Gamma_{\bk_{\ell}}(E))$, and the result follows directly from this and Corollary \ref{cor-a1}.
\end{proof}

\subsection{Existence and decay properties of the guided states}

In light of Lemma \ref{lm-sw2} we examine separately the two cases $\bk_{\ell} \in \mathcal{C}_g(E)$ and $\bk_{\ell} \in \cB \backslash \cB_E$.

\subsubsection{Guided states associated to $\bk_{\ell} \in \mathcal{C}_g(E)$}

The case of $\bk_{\ell} \in \mathcal{C}_g(E)$ is treated by the following:
\begin{lemma}
\label{lm-sw3}
Let $E$, $g$ and $W$ be as in Theorem \ref{thm-a}.
Then for all $\bk_{\ell} \in \mathcal{C}_g(E)$  there exists at least one guided state with energy $E$ associated to $\bk_{\ell}$. Moreover it belongs to $\cap_{m \in {\mathbb N}} \cH_m$.
\end{lemma}
\begin{proof}
Let $\bk_{\ell} \in \mathcal{C}_g(E)$. The spectrum of $\Gamma_{\bk_{\ell}}(E)$ being pure point by Lemma \ref{lm-r1}(a), it follows from the very definition of $\mathcal{C}_g(E)$ that there exists $v \in \cH \backslash \{ 0 \}$ satisfying $(1 - g \Gamma_{\bk_{\ell}}(E)) v = 0$. This entails that $v=g W R_0(\bk_{\ell},E) W v$ according to \eqref{r4} and Proposition \ref{pr-lapz}, and shows that $R_0(\bk_{\ell},E) W v \not=0$. Further, by left multiplying the preceeding equality by $W$, we deduce from \eqref{r4} that $(1- g W^2 R_0(\bk_{\ell},E)) W v = 0$. From this and the obvious identity $(H(\bk_{\ell})-E) R_0(\bk_{\ell},E) W v = (1- g W^2 R_0(\bk_{\ell},E) )Wv$ then follows that 
$$
(H(\bk_{\ell}) -E)  R_0(\bk_{\ell},E) W v =0.
$$
Hence $R_0(\bk_{\ell},E) W v$ is a guided state with energy $E$ associated to $\bk_{\ell}$, which proves the first part of the result.

To show the second part of the statement, we consider a guided state $u \in \cH_{\tau}$, $\tau \geq 0$, associated to $\bk_{\ell} \in \mathcal{C}_g(E)$. Since $u$ is solution to the equation $u = g R_0(\bk_{\ell},E) W^2 u$ by Lemma \ref{lm-sw1}, and $f=W^2 u \in \cH_{3+\epsilon+\tau}$ according to \eqref{rr0}, we deduce from Lemma \ref{lm-nz} that $u=g R_0(\bk_{\ell},E) f \in {\rm H}^{2,3+\epsilon+\tau}(\re \times \cS)$, where we used the notation
${\rm H}^{2,\sigma}(\re \times \cS):= {\rm H}^2(\re \times \cS,(1+x_1^2)^{\sigma \slash 2} {\rm d} x)$ for $\sigma \in \re$. Any guided state associated to $\bk_{\ell}$ and taken in $\cH_{\tau}$ thus belongs to $\cH_{3+\epsilon+\tau}$. From this and the fact that $u \in \cH=\cH_0$ then follows that
$u \in \cH_{\nu}$ for any $\nu \geq 0$. This terminates the proof.
\end{proof}

\subsubsection{Guided states associated to $\bk_{\ell} \in \cB \backslash \cB_E$}
\label{sec-swc}
We now examine the case where $| \bk_{\ell} |^2=E$. We first establish two lemmas, both refering to the notations introduced in \S \ref{sec-spedeco}, and then derive the claim of Theorem \ref{thm-sw} in this particular case.


\begin{lemma}
\label{lm-sw4}
Let $E$, $g$ and $W$ be the same as in Theorem \ref{thm-a}.
Then any guided state $u \in \cH_{\tau}$, $\tau \geq 0$, with energy $E$ and associated to $\bk_{\ell}$, verifies
$\widetilde{W^2 u}(0,\bk_{\ell}) =  (\widetilde{W^2 u})'(0,\bk_{\ell})=0$. Moreover the mapping
$\xi \mapsto  \xi^{-2} \widetilde{W^2 u}(\xi,\bk_{\ell})$ belongs to ${\rm H}^{1+\epsilon+\tau}(\re) \cap {\rm L}_{\rm loc}^1(\re)$.
\end{lemma}
\begin{proof}
The guided state $u$ being solution to the equation $(H_0(\bk_{\ell})-E ) u = g W^2 u$, we get that
\bel{sw0}
\widetilde{W^2 u}(\xi,\bk_{\ell})=g^{-1} \xi^2  \tilde{u}(\xi,\bk_{\ell}),\ {\rm a.e.}\ \xi \in \re,
\ee
directly from \eqref{hzan4}. Further, $\widetilde{W^2 u}(.,\bk_{\ell}) \in {\rm H}^{3+\epsilon+\tau}(\re)$, since this is the Fourier transform over $\re$, according to \eqref{hzan2}-\eqref{hzan3}, of the
${\rm L}^{2,3+\epsilon+\tau}(\re)$-function $x_1 \mapsto \langle (W^2 u)(x_1,.) , \varphi(\bk_{\ell}) \rangle_{{\rm L}^2(\cS)}$.
Hence we have $\widetilde{W^2 u}(.,\bk_{\ell}) \in {\rm C}^1(\re)$ and consequently
\bel{sw0aa}
\widetilde{W^2 u}(\xi,\bk_{\ell}) = \widetilde{W^2 u}(0,\bk_{\ell}) + \xi (\widetilde{W^2 u})'(0,\bk_{\ell}) + \xi \zeta_{\bk_{\ell}}(\xi),\ \xi \in \re,
\ee
for some $\zeta_{\bk_{\ell}} \in {\rm C}^0(\re)$ satisfying $\lim_{\xi \rightarrow 0} \zeta_{\bk_{\ell}}(\xi)=0$.
From \eqref{sw0}, the square integrability of $\tilde{u}(.,\bk_{\ell})$ over $\re$, and the continuity of $\xi \mapsto \widetilde{W^2 u}(\xi,\bk_{\ell})$ at $0$ then follows that
$$\widetilde{W^2 u}(0,\bk_{\ell}) = \lim_{h \downarrow 0} \frac{1}{2h}  \int_{-h}^{h} \widetilde{W^2 u}(\xi,\bk_{\ell}) {\rm d} \xi = g^{-1} \lim_{h \downarrow 0} \frac{1}{2h} \int_{-h}^{h} \xi^2 \tilde{u}(\xi,\bk_{\ell}) {\rm d} \xi=0.$$
This and \eqref{sw0aa} immediately yield
$$ (\widetilde{W^2 u})'(0,\bk_{\ell}) = -g^{-1} \lim_{h \downarrow 0} \frac{1}{2h} \int_{-h}^{h} (\xi \tilde{u}(\xi,\bk_{\ell})+\zeta_{\bk_{\ell}}(\xi)) {\rm d} \xi = 0.$$ 
Finally the last part of the claim is obtained from the two above identities by applying \cite{CD2}[Corollary 3.1] (see also \cite{Agmon}[Lemma B.2]) to $\widetilde{W^2 u}(.,\bk_{\ell})$.
\end{proof}

\begin{lemma}
\label{lm-sw5}
Assume that $E$, $g$ and $W$ are the same as in Theorem \ref{thm-a}.
Then every guided state $u \in \cH_{\tau}$, $\tau \geq 0$, with energy $E$ associated to $\bk_{\ell}$ verifies:
$$u = g R_0(\bk_{\ell},E \pm {\rm i} 0) W^2 u  \in \cH_{1+\epsilon+\tau}. $$
\end{lemma}
\begin{proof}
Let us first prove that $u= g R_0(\bk_{\ell},E \pm {\rm i} 0) W^2 u$. 
With reference to \eqref{sw1} (which still holds true for $|\bk_{\ell}|^2=E$) and Proposition \ref{pr-lapz}(b), it is enough to show that $\lim_{\varepsilon \downarrow 0} R_0(\bk_{\ell},E \pm {\rm i} \varepsilon) W^2 u$ exists in the weak topology of $\cH$. In view of \eqref{hzan5} this can be achieved by establishing for every $v \in \cH$ that
$$ \lim_{\varepsilon \downarrow 0} \sum_{\bK_{\ell} \in \cL^{\perp}} \int_{\re} \left( \frac{\widetilde{W^2 u}(\xi,\bk_{\ell}+\bK_{\ell})}{\xi^2 + | \bk_{\ell} + \bK_{\ell} |^2 -E \mp {\rm i} \varepsilon} - \frac{\widetilde{W^2 u}(\xi,\bk_{\ell}+\bK_{\ell})}{\xi^2 + | \bk_{\ell} + \bK_{\ell} |^2 -E} \right) \overline{\tilde{v}(\xi,\bk_{\ell}+\bK_{\ell})} {\rm d} \xi = 0. $$
Actually, due to \eqref{r6b} and \eqref{hzan4b}, only the case $\bk_{\ell}=0$ has to be examined. This can be done by noticing that
$$ \left| \frac{\widetilde{W^2 u}(\xi,\bk_{\ell})}{\xi^2 \mp {\rm i} \varepsilon} - \frac{\widetilde{W^2 u}(\xi,\bk_{\ell})}{\xi^2} \right| \leq \left| \frac{\widetilde{W^2 u}(\xi,\bk_{\ell})}{\xi^2} \right|,\ {\rm a.e.}\ \xi \in \re, $$
and then invoking the second statement of Lemma \ref{lm-sw4} in order to apply the dominated convergence theorem to the integral $\int_{\re}  ((\xi^2 \mp {\rm i} \varepsilon)^{-1} - \xi^{-2}) \widetilde{W^2 u}(\xi,\bk_{\ell}) \overline{\tilde{v}(\xi,\bk_{\ell})} {\rm d} \xi$.\\
The second part of the proof involves showing that $R_0(\bk_{\ell},E \pm {\rm i} 0) W^2 u  \in \cH_{1+\epsilon+\tau}$.
To do that we decompose $w=W^2 u$ into the sum $w=w_1+w_2$, where $w_1$ verifies
$$\tilde{w}_1(.,\bk_{\ell})=\tilde{w}(.,\bk_{\ell})=\widetilde{W^2 u}(.,\bk_{\ell})\
{\rm and}\ \tilde{w}_1(.,\bk_{\ell}+\bK_{\ell})=0,\ \bK_{\ell} \in \cL^{\perp} \backslash \{ 0 \}.$$
Arguing as in the first part of the proof, we find that
$\lim_{\varepsilon \downarrow 0} R_0(\bk_{\ell},E \pm {\rm i} \varepsilon) w_1 = R_0(\bk_{\ell},E \pm {\rm i} 0) w_1=R_0(\bk_{\ell},E) w_1$ in $\cH$ weakly, with, according to \eqref{hzan2}-\eqref{hzan3},
$$ (R_0(\bk_{\ell},E \pm {\rm i} 0)w_1)(x_1,\bx_{\ell}) = \frac{1}{(2 \pi)^{1 \slash 2} | \cS |^{1 \slash 2}} \left( \int_{\re} {\rm e}^{{\rm i} \xi x_1} \frac{\widetilde{W^2 u}(\xi,\bk_{\ell})}{\xi^2} {\rm d} \xi \right) {\rm e}^{{\rm i} \langle \bk_{\ell},\bx_{\ell} \rangle}. $$
Since $x_1 \mapsto \int_{\re} {\rm e}^{{\rm i} \xi x_1} \xi^{-2}\widetilde{W^2 u}(\xi,\bk_{\ell}) {\rm d} \xi \in {\rm L}^{2,1+\epsilon+\tau}(\re)$ from the second claim of Lemma \ref{lm-sw4}, we deduce from the above identity that $R_0(\bk_{\ell},E \pm {\rm i} 0)w_1 \in \cH_{1+\epsilon+\tau}$.
To prove that $R_0(\bk_{\ell},E \pm {\rm i} 0)w_2 =R_0(\bk_{\ell},E)w_2 \in \cH_{1+\epsilon+\tau}$, we 
notice from \eqref{r6b} that
$$ \xi^2 + | \bk_{\ell} + \bK_{\ell} |^2 - E \geq \xi^2 + \delta E_{\delta} \geq c (\xi^2 + 1),\ \xi \in \re,\ \bK_{\ell} \in \cL^{\perp} \backslash \{ 0 \}, $$
where $c=\min(1,\delta E_{\delta})>0$, then we refer to \eqref{hzan4}-\eqref{hzan4b} and find that
\beas
\| w_2 \|_{\cH}^2 & = & \| ( H_0(\bk_{\ell}) - E) R_0(\bk_{\ell},E) w_2 \|_{\cH}^2  \\
& \geq & c^2 \sum_{\bK_{\ell} \in \cL^{\perp} \backslash \{ 0 \}} \int_{\re} (\xi^2 + 1)^2  \frac{|\tilde{W^2 u}(\xi,\bk_{\ell}+\bK_{\ell})|^2}{(\xi^2 + | \bk_{\ell} + \bK_{\ell} |^2 - E)^2}{\rm d} \xi \\
& \geq & c^2 \| (-\Delta+1) R_0(\bk_{\ell},E) w_2 \|_{\cH}^2.
\eeas
Henceforth we have
\bel{sw3}
\| R_0(\bk_{\ell},E) w_2 \|_{{\rm H}^2(\re \times \cS)} \leq \tilde{c} \| w_2 \|_{\cH} \leq \tilde{c} \| W^2 u \|_{\cH},
\ee
for some constant $\tilde{c}>0$ independent of $u$, according to \eqref{hzan4b}. Further, $W^2 u$ being in  $\cH_{3+\epsilon+\tau}$, we have
$\widetilde{W^2 u}(.,\bk_{\ell}) \in {\rm H}^{3+\epsilon+\tau}(\re)$, whence
$$ w_1(x_1,\bx_{\ell})= \frac{1}{(2 \pi)^{1 \slash 2} | \cS |^{1 \slash 2}} \left( \int_{\re} {\rm e}^{{\rm i} \xi x_1} \widetilde{W^2 u}(\xi,\bk_{\ell}){\rm d} \xi \right) {\rm e}^{{\rm i} \langle \bk_{\ell},\bx_{\ell} \rangle} \in \cH_{3 + \epsilon+\tau}, $$
and consequently $w_2=W^2 u - w_1 \in \cH_{3 + \epsilon+\tau}$. Finally, arguing as in the proof of Lemma \ref{lm-nz} and using \eqref{sw3}, we obtain that $R_0(\bk_{\ell},E) w_2 \in \cH_{3+\epsilon+\tau}$.
\end{proof}

Any guided state $u$ being in $\cH$, the claim of Theorem \ref{thm-sw} follows by successively applying Lemma \ref{lm-sw5} to $u$. 

\subsection{Proof of Corollary \ref{cor-sw}}
\label{sec-cor-sw}
We now establish the coming lemma, which, together with Theorem \ref{thm-sw}, entails Corollary \ref{cor-sw}.

\begin{lemma}
\label{lm-sw6}
Let $E$ and $W$ be the same as in Corollary \ref{cor-sw}. Then for every $\bk_{\ell} \in \cB$ such that
$| \bk_{\ell} |=E^{1 \slash 2}$, we have $\ker( H(\bk_{\ell}) - E ) = \{ 0 \}$.
\end{lemma}
\begin{proof}
Let $u \in D(H(\bk_{\ell})) = D(H_0(\bk_{\ell}))$ be solution to the eigenvalue equation 
\bel{fo1}
H(\bk_{\ell}) u = (-\Delta - g W^2) u = E u.
\ee
With reference to \eqref{hzan2}, the set $\{ \varphi(\bk_{\ell}+\bK_{\ell}) \}_{ \bK_{\ell} \in \cL^{\perp} }$ is an orthonormal basis of ${\rm L}^2(\cS)$, so $u$ decomposes as
\bel{fo2}
u(x_1,\bx_{\ell}) = \sum_{\bK_{\ell} \in \cL^{\perp}} u_{\bk_{\ell}+\bK_{\ell}}(x_1) \varphi(\bk_{\ell}+\bK_{\ell};\bx_{\ell}),
\ee
where $u_{\bk_{\ell}+\bK_{\ell}}(x_1):=\int_{\cS} u(x_1,\bx_{\ell}) \overline{\varphi(\bk_{\ell}+\bK_{\ell};\bx_{\ell})} {\rm d} \bx_{\ell}$. From this and \eqref{fo1} then follows that 
\bel{fo3}
-u_{\bk_{\ell}+\bK_{\ell}}''(x_1) + ( | \bk_{\ell} + \bK_{\ell} |^2 - E ) u_{\bk_{\ell}+\bK_{\ell}}(x_1) - g (W^2 u)_{\bk_{\ell}+\bK_{\ell}}(x_1) = 0,\ x_1 \in \re,\ \bK_{\ell} \in \cL^{\perp}.
\ee
Further we have
$W^2(x_1,\bx_{\ell}) = \sum_{\bK_{\ell}' \in \cL^{\perp}} (W^2)_{\bK_{\ell}'}(x_1) \varphi(\bK_{\ell}';\bx_{\ell})$ for a.e. $(x_1,\bx_{\ell}) \in \re \times \cS$,
from where we get that
\bea
(W^2 u)_{\bk_{\ell}+\bK_{\ell}}(x_1) & = & \int_{\cS} W(x_1,\bx_{\ell})^2 u(x_1,\bx_{\ell}) \overline{\varphi(\bk_{\ell}+\bK_{\ell};\bx_{\ell})} {\rm d} \bx_{\ell} \nonumber \\
& = & \sum_{\bK_{\ell}' \in \cL^{\perp}} (W^2)_{\bK_{\ell}'}(x_1) \int_{\cS} u(x_1,\bx_{\ell}) \varphi(\bK_{\ell}';\bx_{\ell}) \overline{\varphi(\bk_{\ell}+\bK_{\ell};\bx_{\ell})} {\rm d} \bx_{\ell} \nonumber \\
& =& | \cS |^{-1 \slash 2} \sum_{\bK_{\ell}' \in \cL^{\perp}} (W^2)_{\bK_{\ell}'}(x_1) u_{\bk_{\ell}+\bK_{\ell}-\bK_{\ell}'}(x_1), \label{fo5}
\eea
by using the fact that $\varphi(\bK_{\ell}') \overline{\varphi(\bk_{\ell}+\bK_{\ell})}
=| \cS |^{-1 \slash 2} \varphi(\bk_{\ell}+\bK_{\ell}-\bK_{\ell}')$.
From \eqref{fo5} then follows for a.e. $x_1 \in \re$ that 
$$
\left( \sum_{\bK_{\ell} \in \cL^{\perp}} | (W^2 u)_{\bk_{\ell}+\bK_{\ell}}(x_1) |^2 \right)^{1 \slash 2}
\leq | \cS |^{-1 \slash 2} \left( \sum_{\bK_{\ell} \in \cL^{\perp}} | (W^2)_{\bK_{\ell}}(x_1) | \right)
\left( \sum_{\bK_{\ell} \in \cL^{\perp}} | u_{\bk_{\ell}+\bK_{\ell}}(x_1) |^2 \right)^{1 \slash 2}.
$$
In light of (ii) we may deduce from this that there is a constant $c(W,\cS)>0$, depending only on $W$ and $| \cS |$,
such that
\bel{fo6}
\sum_{\bK_{\ell} \in \cL^{\perp}} \| (W^2 u)_{\bk_{\ell}+\bK_{\ell}} \|_{{\rm L}^2(\re)}^2 \leq c(W,\cS)^2 \sum_{\bK_{\ell} \in \cL^{\perp}} \| u_{\bk_{\ell}+\bK_{\ell}} \|_{{\rm L}^2(\re)}^2.
\ee 
This boils down to the fact that
$$ | (W^2)_{\bK_{\ell}}(x_1) | \leq \frac{| \cS |^{1 \slash 2}}{(1+| \bK_{\ell} |^2)^2} \| (1-\Delta_{\ell})^2 W^2 \|_{{\rm L}^{\infty}(\re \times \cS)},\ x_1 \in \re,\ \bK_{\ell} \in \cL^{\perp}, $$
where $\Delta_{\ell}$ denotes the Laplace operator w.r.t. to $\bx_{\ell}$, as can be seen by noticing that 
$\Delta_{\ell} \varphi(\bK_{\ell}) = -| \bK_{\ell} |^2 \varphi(\bK_{\ell})$, and integrating by parts in the following integral:
\beas
(W^2)_{\bK_{\ell}}(x_1) & = & \int_{\cS} W(x_1,\bx_{\ell})^2 \frac{1}{(1+| \bK_{\ell} |^2)^2} \overline{(1-\Delta_{\ell})^2 \varphi(\bK_{\ell};\bx_{\ell})} {\rm d} x_{\ell} \\
& =& \frac{1}{(1+| \bK_{\ell} |^2)^2} \int_{\cS} \overline{\varphi(\bK_{\ell};\bx_{\ell})} (1-\Delta_{\ell})^2 W^2(x_1,\bx_{\ell}) {\rm d} \bx_{\ell}.
\eeas
Further, for all $\bK_{\ell} \in \cL^{\perp}$ it follows from (i) that $(W^2)_{\bK_{\ell}}(x_1)=\int_{\cS} W(x_1,\bx_{\ell})^2 \overline{\varphi(\bK_{\ell};\bx_{\ell})} {\rm d} \bx_{\ell}=0$ for a.e. $x_1 \in I$, which, together with \eqref{fo5}, entails $(W^2 u)_{\bk_{\ell}+\bK_{\ell}}(x_1)=0$. In light of \eqref{fo3} for $\bK_{\ell}=0$, this yields 
$-u_{\bk_{\ell}}''(x_1)=0$ for a.e. $x_1 \in I$, hence $u_{\bk_{\ell}}(x_1)=0$ since 
$u_{\bk_{\ell}}$ is square integrable in $I$, and finally 
\bel{fo7}
u_{\bk_{\ell}}(x_1)=0,\ {\rm a.e.}\ x_1 \in \re. 
\ee
by applying Cauchy's theorem.\\
\noindent The next step of the proof involves multiplying \eqref{fo3} by $\overline{u_{\bk_{\ell}+\bK_{\ell}}(x_1)}$ for each $\bK_{\ell} \in \cL^{\perp}$, integrating the result w.r.t. $x_1$ over $\re$, and then performing an integration by parts in the first integral in the l.h.s. of the obtained equality. We find that
$$
\| u_{\bk_{\ell}+\bK_{\ell}}' \|_{{\rm L}^2(\re)}^2 + ( | \bk_{\ell} + \bK_{\ell} |^2 - E ) \| u_{\bk_{\ell}+\bK_{\ell}} \|_{{\rm L}^2(\re)}^2  - g \langle (W^2u)_{\bk_{\ell}+\bK_{\ell}} ,  u_{\bk_{\ell}+\bK_{\ell}} \rangle_{{\rm L}^2(\re)}=0,\ \forall \bK_{\ell} \in \cL^{\perp}.
$$
Summing up the above identity over $\bK_{\ell} \in \cL^{\perp} \backslash \{ 0 \}$, we deduce from \eqref{r6b} and \eqref{fo6}-\eqref{fo7} that
$$ \sum_{\bK_{\ell} \in \cL^{\perp} \backslash \{ 0 \}} \left( \| u_{\bk_{\ell}+\bK_{\ell}}' \|_{{\rm L}^2(\re)}^2
+ (\delta E_{\delta} - g c(W,\cS) ) \| u_{\bk_{\ell}+\bK_{\ell}} \|_{{\rm L}^2(\re)}^2 \right) \leq 0. $$
As a consequence we have $u_{\bk_{\ell}+\bK_{\ell}}=0$ in ${\rm L}^2(\re)$ for every $\bK_{\ell} \in \cL^{\perp} \backslash \{ 0 \}$, provided $g \in (0,\delta E_{\delta} \slash c(W,\cS))$. From this, \eqref{fo2} and \eqref{fo7} then follows that $u=0$, which proves the result.
\end{proof}

\appendix

\label{sec-app}

\section{Appendix A: Proof of Lemma \ref{lm-r1}}
\label{sec-applmr1}
We prove the statements (a), (b) and (c) of  Lemma \ref{lm-r1} successively, the claims (d) and (e) being treated simultaneously.

\noindent (a) Fix $\varepsilon$ in $\re$. We first prove the result for $\Gamma_{\bk_{\ell}}(E+{\rm i} \varepsilon)$ and $\bk_{\ell} \in \cB_E^+$, the cases of $\bk_{\ell} \in \cB_E^-$ and $C_{\bk_{\ell}}(E+{\rm i} \varepsilon)$ being treated in a similar way. To do that we refer to \eqref{r6b}-\eqref{r6cg3}, and write
\bel{r6star}
| \gamma_{\bk_{\ell}}(\bx,\by,E+{\rm i} \varepsilon) |^2
= \beta^{2} W(\bx)^2 W(\by)^2 \zeta_{\bk_{\ell}}(\bx,\by,E+{\rm i} \varepsilon),
\ee
where
\bel{r6star1}
\zeta_{\bk_{\ell}}(\bx,\by,E+{\rm i} \varepsilon)
:= \sum_{\bK_{\ell},\bK_{\ell}' \in \cL^{\perp}}
\frac{{\rm e}^{{\rm i} ( p(\bk_{\ell}+\bK_{\ell},E+{\rm i} \varepsilon) - \overline{p(\bk_{\ell}+\bK_{\ell}',E+{\rm i} \varepsilon )}) |x_1 - y_1|}}{p(\bk_{\ell}+\bK_{\ell},E+{\rm i} \varepsilon)  \overline{p(\bk_{\ell}+\bK_{\ell}',E+{\rm i} \varepsilon)}}
{\rm e}^{{\rm i} \langle \bK_{\ell}-\bK_{\ell}', \bx_{\ell}-\by_{\ell} \rangle}.
\ee
In light of \eqref{rr4}-\eqref{r20}, and since $p_I(\bk_{\ell}+\bK_{\ell},E) >0$ for every $(\bk_{\ell},\bK_{\ell}) \in \cB_E^+ \times \cL^{\perp}$ from \eqref{rr1}, the series $\sum_{\bK_{\ell},\bK_{\ell}' \in \cL^{\perp}}
\frac{{\rm e}^{- ( p_I(\bk_{\ell}+\bK_{\ell},E+{\rm i} \varepsilon) + p_I(\bk_{\ell}+\bK_{\ell}',E+{\rm i} \varepsilon )) |x_1 - y_1|}}{p_I(\bk_{\ell}+\bK_{\ell},E+{\rm i} \varepsilon) p_I(\bk_{\ell}+\bK_{\ell}',E+{\rm i} \varepsilon)}$ converges for every $x_1 \neq y_1$ in $\re$, hence the series in the r.h.s. of \eqref{r6star1}
is normally convergent on $\cS \times \cS$. Bearing in mind that
\bel{r6e}
\int_{\cS} {\rm e}^{{\rm i} \langle \bK_{\ell}-\bK_{\ell}', \bx_{\ell} \rangle} {\rm d} \bx_{\ell} = | \cS | \delta(\bK_{\ell},\bK_{\ell}'),\ \bK_{\ell}, \bK_{\ell}' \in \cL^{\perp},
\ee
this entails
\bel{r6cter}
\int_{\cS \times \cS} \zeta_{\bk_{\ell}}(\bx,\by,E+{\rm i} \varepsilon)  {\rm d} \bx_{\ell} {\rm d} \by_{\ell}
= | \cS |^2 \sum_{\bK_{\ell}\in \cL^{\perp}} \frac{{\rm e}^{- 2p_I(\bk_{\ell}+\bK_{\ell},E+{\rm i} \varepsilon) |x_1 - y_1|}}{| p(\bk_{\ell}+\bK_{\ell},E+{\rm i} \varepsilon)|^2} := | \cS |^2 f_{\varepsilon}(x_1,y_1).
\ee
Further, we have
$\zeta_{\bk_{\ell}}(\bx,\by,E+{\rm i} \varepsilon) \in \re_+$ for all $\bx, \by \in \re \times \cS$, from \eqref{r6star},
hence
\bel{r6ebis}
\int_{\cS \times \cS} | \gamma_{\bk_{\ell}}(\bx,\by,E+{\rm i} \varepsilon) |^2 {\rm d} \bx_{\ell} {\rm d} \by_{\ell}
\leq  \beta^2 | \cS |^2 \| W(x_1,.) \|_{{\rm L}^{\infty}(\cS)}^2 \| W(y_1,.) \|_{{\rm L}^{\infty}(\cS)}^2 f_{\varepsilon}(x_1,y_1),
\ee
for each $x_1  \neq y_1$ in $\re$, by \eqref{r6cter}.
Moreover the mapping $y_1 \mapsto f_{\varepsilon}(x_1,y_1)$ is integrable on $\re$ for every $x_1$ in $\re$, with
\bel{r6d}
\int_{\re} f_{\varepsilon}(x_1,y_1) {\rm d} y_1 \leq
\sum_{\bK_{\ell} \in \cL^{\perp}} \frac{1}{p_I(\bk_{\ell}+\bK_{\ell},E+{\rm i} \varepsilon)^3}:=\alpha_3(\delta,\varepsilon),
\ee
the series in the r.h.s. of \eqref{r6d} being convergent according to \eqref{r20} and Lemma \ref{lm-r3}. Henceforth
$\int_{\re} \| W(y_1,.) \|_{{\rm L}^{\infty}(\cS)}^2 f_{\varepsilon}(x_1,y_1) {\rm d} y_1 \leq
\| W \|_{{\rm L}^{\infty}(\re \times \cS)}^2 \alpha_3(\delta,\varepsilon)$
for all $x_1 \in \re$, whence
$$
\int_{\re \times \re} \| W(x_1,.) \|_{{\rm L}^{\infty}(\cS)}^2  \| W(y_1,.) \|_{{\rm L}^{\infty}(\cS)}^2 f_{\varepsilon}(x_1,y_1) {\rm d} x_1 {\rm d} y_1 \leq
\| W \|_{{\rm L}^{\infty}(\re \times \cS)}^2 \| W \|_{{\rm L}^{2}(\re,{\rm L}^{\infty}(\cS))}^2
\alpha_3(\delta,\varepsilon).
$$
From this and \eqref{r6ebis} then follows that
\bel{r6f}
\int_{(\re \times \cS)^2} | \gamma_{\bk_{\ell}}(\bx,\by,E+ {\rm i} \varepsilon) |^2  {\rm d} \bx {\rm d} \by \leq
\beta^2 | \cS |^2 \| W \|_{{\rm L}^{2}(\re,{\rm L}^{\infty}(\cS))}^2 \| W \|_{{\rm L}^{\infty}(\re \times \cS)}^2
\alpha_3(\delta,\varepsilon),
\ee
by the Tonelli-Fubini theorem.\\
\noindent (b) The second part of the result boils down to the fact that
$$\overline{\gamma_{\bk_{\ell}}(\bx,\by,E+{\rm i} \varepsilon)}= \gamma_{\bk_{\ell}}(\by,\bx,E-{\rm i} \varepsilon),\ (\bk_{\ell},\varepsilon) \in ( \cB_E^- \times \re^*) \cup ( \cB_E^+ \times \re),\ \bx,\by \in \re \times \cS, $$
as
$\overline{p(\bk_{\ell}+\bK_{\ell},E+{\rm i} \varepsilon)}=-p(\bk_{\ell}+\bK_{\ell},E-{\rm i} \varepsilon)$ for every $\bK_{\ell} \in \cL^{\perp}$ by \eqref{r6cg3} and \eqref{rr1}-\eqref{rr2}.\\
\noindent (c) In light of \eqref{r6cg} and \eqref{r9g}, $c_{\bk_{\ell}}(\bx,\by,E+{\rm i} \varepsilon)$ decomposes into the sum
\bel{r13a}
c_{\bk_{\ell}}(\bx,\by,E+{\rm i} \varepsilon) := \sum_{j=1,2} c_{\bk_{\ell}}^{(j)}(\bx,\by,E+{\rm i} \varepsilon),
\ee
where
\bel{r13b}
c_{\bk_{\ell}}^{(1)}(\bx,\by,E+{\rm i} \varepsilon) := \beta W(\bx) W(\by)
\frac{{\rm e}^{{\rm i} p(\bk_{\ell},E+{\rm i} \varepsilon) |x_1 - y_1|}-1}{-{\rm i} p(\bk_{\ell},E+{\rm i} \varepsilon)} {\rm e}^{{\rm i} \langle \bk_{\ell}, \bx_{\ell}-\by_{\ell} \rangle},
\ee
and
\bel{r13c}
c_{\bk_{\ell}}^{(2)}(\bx,\by,E+{\rm i} \varepsilon) := \beta W(\bx) W(\by) \sum_{\bK_{\ell} \in \cL^{\perp} \backslash \{ 0 \}}
\frac{{\rm e}^{{\rm i} p(\bk_{\ell}+\bK_{\ell},E+{\rm i} \varepsilon) |x_1 - y_1|}}{-{\rm i} p(\bk_{\ell}+\bK_{\ell},E+{\rm i} \varepsilon)}  {\rm e}^{{\rm i} \langle \bk_{\ell}+\bK_{\ell}, \bx_{\ell}-\by_{\ell} \rangle}.
\ee
Let $C_{\bk_{\ell}}^{(j)}(E+{\rm i} \varepsilon)$, $j=1,2$, denote the integral operator with kernel $c_{\bk_{\ell}}^{(j)}(\bx,\by,E+{\rm i} \varepsilon)$.
We shall compute the Hilbert-Schmidt norm of each $C_{\bk_{\ell}}^{(j)}(E+{\rm i} \varepsilon)$, $j=1,2$, successively.
First, we have
$$
| c_{\bk_{\ell}}^{(1)}(\bx,\by,E+{\rm i} \varepsilon) |^2
= \beta^2 W(\bx)^2 W(\by)^2 \frac{| {\rm e}^{{\rm i} p(\bk_{\ell},E+{\rm i} \varepsilon) |x_1 - y_1|} -1 |^2}{| p(\bk_{\ell},E+{\rm i} \varepsilon)|^2},\ \bx, \by \in \re \times \cS,
$$
by \eqref{r13b}, whence
$| c_{\bk_{\ell}}^{(1)}(\bx,\by,E+{\rm i} \varepsilon) |^2
\leq \beta^2 W(\bx)^2 W(\by)^2 |x_1-y_1|^2$ from \eqref{rr1} together with the inequality
$|{\rm e}^{z} -1| \leq |z|$, which holds true for all complex number $z$ such that $\Pre{z} \leq 0$.
This entails $| c_{\bk_{\ell}}^{(1)}(\bx,\by,E+{\rm i} \varepsilon) |^2
\leq 2 \beta^2 W(\bx)^2 W(\by)^2 (x_1^2+y_1^2)$ for all $\bx$ and $\by$ in $\re \times \cS$, and hence
\bel{r14}
\int_{(\re \times \cS)^2} | c_{\bk_{\ell}}^{(1)}(\bx,\by,E+{\rm i} \varepsilon) |^2 {\rm d} \bx {\rm d} \by
\leq 4 \beta^2 \| W \|_{\cH_1}^2,
\ee
by recalling \eqref{r2b}.
Further, arguing in the same way as in the derivation of \eqref{r6f}, we deduce from \eqref{r13c} that
$$
\int_{(\re \times \cS)^2}  | c_{\bk_{\ell}}^{(2)}(\bx,\by,E+{\rm i} \varepsilon) |^2 {\rm d} \bx {\rm d} \by
\leq \beta^2 | \cS |^2  \| W \|_{{\rm L}^{2}(\re, {\rm L}^{\infty}( \cS))}^2 \| W \|_{{\rm L}^{\infty}(\re \times \cS)}^2 \alpha_3(\delta),
$$
where $\alpha_3(\delta)$ is the constant defined in Lemma \ref{lm-r3}. Now (c) follows from this, \eqref{r13a} and \eqref{r14}.\\
\noindent (d) \& (e)  In light of \eqref{r13b} we may write for every $\bk_{\ell} \in \cB_E$ and $\varepsilon \in \re$,
\bel{es1}
| c_{\bk_{\ell}}^{(1)}(\bx,\by,E+{\rm i} \varepsilon) - c_{\bk_{\ell}}^{(1)}(\bx,\by,E) |^2
= \beta^2 W(\bx)^2 W(\by)^2 \zeta_{\bk_{\ell}}^{(1)}(\bx,\by,E + {\rm i} \varepsilon),\ \bx, \by \in \re \times \cS,
\ee
with, due to \eqref{r6e},
\bel{es2}
\int_{\cS^2} \zeta_{\bk_{\ell}}^{(1)}(\bx,\by,E + {\rm i} \varepsilon) {\rm d} \bx_{\ell} {\rm d} \by_{\ell}
= | \cS |^2 \sum_{\bK_{\ell} \in \cL^{\perp} \backslash \{ 0 \} } | d_{\bk_{\ell}+\bK_{\ell}}(x_1,y_1,E+{\rm i} \varepsilon) |^2,
\ee
where we have set for each $\bK_{\ell} \in \cL^{\perp} \backslash \{ 0 \}$,
\bel{es3}
d_{\bk_{\ell}+\bK_{\ell}}(x_1,y_1,E+{\rm i} \varepsilon) :=
\frac{{\rm e}^{{\rm i} p(\bk_{\ell}+\bK_{\ell},E+{\rm i} \varepsilon) | x_1-y_1|}}{-{\rm i} p(\bk_{\ell}+\bK_{\ell},E+{\rm i} \varepsilon)} - \frac{{\rm e}^{- p_I(\bk_{\ell}+\bK_{\ell},E) | x_1-y_1|}}{p_I(\bk_{\ell}+\bK_{\ell},E)}.
\ee
Let us decompose $d_{\bk_{\ell}+\bK_{\ell}}(x_1,y_1,E+{\rm i} \varepsilon)$ into the sum
$\sum_{j=1}^3 d_{\bk_{\ell}+\bK_{\ell}}^{(j)}(x_1,y_1,E+{\rm i} \varepsilon)$, with
\beas
d_{\bk_{\ell}+\bK_{\ell}}^{(1)}(x_1,y_1,E+{\rm i} \varepsilon) & := & \frac{{\rm e}^{-p_I(\bk_{\ell}+\bK_{\ell},E+{\rm i} \varepsilon ) |x_1 - y_1|} ( {\rm e}^{{\rm i} p_R(\bk_{\ell}+\bK_{\ell},E+{\rm i} \varepsilon) |x_1 - y_1|} - 1)}{- {\rm i} p(\bk_{\ell}+\bK_{\ell},E+{\rm i} \varepsilon)}, \\
d_{\bk_{\ell}+\bK_{\ell}}^{(2)}(x_1,y_1,E+{\rm i} \varepsilon) & := & \frac{{\rm e}^{- p_I(\bk_{\ell}+\bK_{\ell},E+{\rm i} \varepsilon) |x_1 - y_1|} - {\rm e}^{-p_I(\bk_{\ell}+\bK_{\ell},E) |x_1 - y_1|}}{- {\rm i} p(\bk_{\ell}+\bK_{\ell},E+{\rm i} \varepsilon)},\\
d_{\bk_{\ell}+\bK_{\ell}}^{(3)}(x_1,y_1,E+{\rm i} \varepsilon) & := & {\rm e}^{-p_I(\bk_{\ell}+\bK_{\ell},E) |x_1 - y_1|} \left(
\frac{1}{- {\rm i} p(\bk_{\ell}+\bK_{\ell},E+{\rm i} \varepsilon)} - \frac{1}{p_I(\bk_{\ell}+\bK_{\ell},E)} \right),
\eeas
and notice for every $(\bk_{\ell},\bK_{\ell}) \in ( \cB_E^+ \times \cL^{\perp} ) \cup ( \cB_E^- \times ( \cL^{\perp} \backslash \{ 0 \}) )$ and $\varepsilon \in \re$, that we have
\bel{es4}
p_I(\bk_{\ell}+\bK_{\ell},E)=|p(\bk_{\ell}+\bK_{\ell},E)| \leq p_I(\bk_{\ell}+\bK_{\ell},E+{\rm i} \varepsilon) \leq p_I(\bk_{\ell}+\bK_{\ell},E) + \left( \frac{ |\varepsilon |}{2} \right)^{1 \slash 2},
\ee
from \eqref{rr1}, and consequently
\bel{es5}
| p(\bk_{\ell}+\bK_{\ell},E+{\rm i} \varepsilon)-p(\bk_{\ell}+\bK_{\ell},E) | \leq \frac{| \varepsilon |}{2 p_I(\bk_{\ell}+\bK_{\ell},E)},
\ee
since $(p(\bk_{\ell}+\bK_{\ell},E+{\rm i} \varepsilon)-p(\bk_{\ell}+\bK_{\ell},E))(p(\bk_{\ell}+\bK_{\ell},E+{\rm i} \varepsilon)+p(\bk_{\ell}+\bK_{\ell},E))= {\rm i} \varepsilon$.
Thus, by using that $| \sin u | \leq u$ for $u \geq 0$ (resp. ${\rm e}^{-u} - {\rm e}^{-v} \leq (v-u) {\rm e}^{-u}$ for $v \geq u$) in the estimation of $d_{\bk_{\ell}+\bK_{\ell}}^{(1)}(x_1,y_1,E+{\rm i} \varepsilon)$ (resp. $d_{\bk_{\ell}+\bK_{\ell}}^{(2)}(x_1,y_1,E+{\rm i} \varepsilon)$), we deduce from \eqref{es4}-\eqref{es5} for every $\varepsilon$, $x_1$ and $y_1$ in $\re$, that
\bel{es6}
| d_{\bk_{\ell}+\bK_{\ell}}(x_1,y_1,E+{\rm i} \varepsilon) | \leq {\rm e}^{-p_I(\bk_{\ell}+\bK_{\ell},E) |x_1 - y_1|} \left( \frac{|x_1- y_1|}{p_I(\bk_{\ell}+\bK_{\ell},E)^2}+ \frac{1}{2 p_I(\bk_{\ell}+\bK_{\ell},E)^3} \right) | \varepsilon |.
\ee
Fix $x_1$ in $\re$. The series $\sum_{\bK_{\ell} \in \cL^{\perp} \backslash \{ 0 \}} | d_{\bk_{\ell}+\bK_{\ell}}(x_1,y_1,E+{\rm i} \varepsilon) |^2$ is normally convergent for each $y_1 \in \re$. This comes from Lemma \ref{lm-r3} and  \eqref{es6}. So we get that  
$$
\int_{\cS \times (\re \times \cS)} \zeta_{\bk_{\ell}}^{(1)}(\bx,\by,E + {\rm i} \varepsilon) {\rm d} \bx_{\ell} {\rm d} \by = \sum_{\bK_{\ell} \in \cL^{\perp} \backslash \{ 0 \}} \int_{\re} | d_{\bk_{\ell}+\bK_{\ell}}(x_1,y_1,E+{\rm i} \varepsilon) |^2 {\rm d} y_1 \leq \frac{3 \alpha_7(\delta)}{2} \varepsilon^2, $$
from \eqref{es2} and \eqref{es6}.
In light of \eqref{es1}-\eqref{es2}, this yields
\bel{es7}
\| C_{\bk_{\ell}}^{(1)}(E + {\rm i} \varepsilon) - C_{\bk_{\ell}}^{(1)}(E) \|_{HS} \leq \left( \frac{3 \alpha_7(\delta)}{2} \right)^{1 \slash 2} \varrho | \varepsilon |,\ \bk_{\ell} \in \cB_E,\varepsilon \in \re,
\ee
where $\varrho:=\beta | \cS | \| W \|_{{\rm L}^{\infty}(\re \times \cS)} \| W \|_{{\rm L}^{2}(\re, {\rm L}^{\infty}(\cS))}$, by straightforward computations. Arguing in the same way, we get moreover that
\bel{es8}
\| C_{\bk_{\ell}}^{(0)}(E + {\rm i} \varepsilon) - C_{\bk_{\ell}}^{(0)}(E) \|_{HS} \leq
\left( \frac{3}{2 p_I(\bk_{\ell},E)^7} \right)^{1 \slash 2} \varrho | \varepsilon |,\ \bk_{\ell} \in \cB_E^+,\ \varepsilon \in \re.
\ee
Notice from \eqref{rr3}-\eqref{rr3b} that the picture is quite different for $\bk_{\ell} \in \cB_E^-$. Indeed, in this case it holds true that $p(\bk_{\ell},E) >0$ by \eqref{rr3}, and, for every $\varepsilon >0$,
\beas
p_R(\bk_{\ell},E \pm {\rm i} \varepsilon) & = & \pm \left( \frac{(p(\bk_{\ell},E)^4 + \varepsilon^2)^{1 \slash 2} + p(\bk_{\ell},E)^2}{2} \right)^{1 \slash 2}, \\
p_I(\bk_{\ell},E \pm {\rm i} \varepsilon) & = & \left( \frac{(p(\bk_{\ell},E)^4 + \varepsilon^2)^{1 \slash 2} - p(\bk_{\ell},E)^2}{2} \right)^{1 \slash 2},
\eeas
by \eqref{rr1}-\eqref{rr2}. As a consequence we have
\bel{es9}
| p(\bk_{\ell},E \pm {\rm i} \varepsilon) \mp p(\bk_{\ell},E) | \leq \varepsilon^{1 \slash 2},\ \bk_{\ell} \in \cB_E^-,\ \varepsilon \in \re_+^*.
\ee
Therefore, $| c_{\bk_{\ell}}^{(0)}(E \pm {\rm i} \varepsilon,\bx,\by) \mp c_{\bk_{\ell}}^{(0)}(E,\bx,\by) |$ being of the form $\beta W(\bx) W(\by) | \tilde{c}_{\bk_{\ell}}(E \pm {\rm i} \varepsilon,x_1,y_1) |$, with
$\tilde{c}_{\bk_{\ell}}(E \pm {\rm i} \varepsilon,x_1,y_1)=\sum_{j=1}^3\tilde{c}_{\bk_{\ell}}^{(j)}(E \pm {\rm i} \varepsilon,x_1,y_1)$, and
\beas
\tilde{c}_{\bk_{\ell}}^{(1)}(E \pm {\rm i} \varepsilon,x_1,y_1) & = &
\frac{{\rm e}^{{\rm i} p( \bk_{\ell}, E \pm {\rm i} \varepsilon) | x_1 - y_1 |} - {\rm e}^{{\rm i} p_R( \bk_{\ell}, E \pm {\rm i} \varepsilon) | x_1 - y_1 |}}{- {\rm i} p( \bk_{\ell}, E \pm {\rm i} \varepsilon)}, \\
\tilde{c}_{\bk_{\ell}}^{(2)}(E \pm {\rm i} \varepsilon,x_1,y_1) & = & \frac{{\rm e}^{{\rm i} p_R( \bk_{\ell}, E \pm {\rm i} \varepsilon) | x_1 - y_1 |} - {\rm e}^{\pm {\rm i} p( \bk_{\ell}, E) | x_1 - y_1 |}}{- {\rm i} p( \bk_{\ell}, E \pm {\rm i} \varepsilon)}, \\
\tilde{c}_{\bk_{\ell}}^{(3)}(E \pm {\rm i} \varepsilon,x_1,y_1) & = & {\rm e}^{\pm {\rm i} p( \bk_{\ell}, E) | x_1 - y_1 |} \left( \frac{1}{- {\rm i} p( \bk_{\ell}, E \pm {\rm i} \varepsilon)} - \frac{1}{- {\rm i} p( \bk_{\ell}, E)} \right),
\eeas
we then derive from \eqref{es9} for every $\bk_{\ell} \in \cB_E^-$, $\varepsilon \in \re_+^*$ and $\bx, \by \in \re \times \cS$, that
$$
| c_{\bk_{\ell}}^{(0)}(E \pm {\rm i} \varepsilon,\bx,\by) \mp c_{\bk_{\ell}}^{(0)}(E,\bx,\by) | \leq  \beta W(\bx) W(\by) \left( 2^{1 \slash 2} |x_1-y_1| + \frac{1}{p(\bk_{\ell},E)} \right) \frac{\varepsilon^{1 \slash 2}}{p(\bk_{\ell},E)}.
$$
Taking account of \eqref{r2b}, this entails
$$
\| C_{\bk_{\ell}}^{(0)}(E \pm {\rm i} \varepsilon) \mp C_{\bk_{\ell}}^{(0)}(E) \|_{HS} \leq \beta | \cS | \left( 2 \frac{\| W \|_{\cH_1}^2}{p(\bk_{\ell},E)} + \frac{1}{p(\bk_{\ell},E)^2} \right) ( 2 \varepsilon)^{1 \slash 2},\ \bk_{\ell} \in \cB_E^-,\ \varepsilon \in \re_+^*,
$$
by direct computations. Now the result follows immediately from this and \eqref{es7}-\eqref{es8}.


\section{Appendix B: LAP for $H_0(\bk_{\ell})$, $\bk_{\ell} \in \cB$}
\label{sec-laphz}

\subsection{Spectral decomposition and generalized Fourier coefficients}
\label{sec-spedeco}
For each $(\xi,\bk_{\ell}) \in \re \times \cB$ and $\bK_{\ell} \in \cL^{\perp}$, we introduce
\bel{hzan1}
\phi(\xi,\bk_{\ell}+\bK_{\ell};\bx) := \frac{1}{(2 \pi)^{1 \slash 2}} {\rm e}^{{\rm i} \xi  x_1} \varphi(\bk_{\ell}+\bK_{\ell};\bx_{\ell}),\  \bx=(x_1,\bx_{\ell}) \in \re \times \cS,
\ee
where
\bel{hzan2}
\varphi(\bk_{\ell}+\bK_{\ell};\bx_{\ell}) := \frac{1}{|\cS|^{1 \slash 2}} {\rm e}^{{\rm i} \langle \bk_{\ell} + \bK_{\ell} , \bx_{\ell} \rangle},\ \bx_{\ell} \in \cS,
\ee
then we define the generalized Fourier coefficient of any $u \in \cH$ as
\bea
\widetilde{u}(\xi,\bk_{\ell}+\bK_{\ell}) &  := &  \lim_{X \rightarrow +\infty} \langle u, \phi(\xi,\bk_{\ell}+\bK_{\ell}) \rangle_{{\rm L}^2((-X,X) \times \cS)} \nonumber \\
&  =& \frac{1}{ (2 \pi)^{1 \slash 2}} \lim_{X \rightarrow +\infty}  \int_{-X}^X {\rm e}^{-{\rm i} \xi  x_1} \langle u(x_1,.), \varphi(\bk_{\ell}+\bK_{\ell}) \rangle_{{\rm L}^2(\cS)} {\rm d} x_1. \label{hzan3}
\eea
For every $\bk_{\ell} \in \cB$ fixed, the set $\{ \phi(\xi,\bk_{\ell}+\bK_{\ell}),\ \bK_{\ell} \in \cL^{\perp},\  \xi  \in \re \}$ is a complete system of generalized eigenfunctions of $H_0(\bk_{\ell})$, in the sense that:
\begin{enumerate}[(a)]
\item $\cF_{\bk_{\ell}} : u \mapsto (\widetilde{u}(.,\bk_{\ell}+\bK_{\ell}))_{\bK_{\ell} \in \cL^{\perp}}$ is a unitary transform from $\cH$ onto $\bigoplus_{\bK_{\ell} \in \cL^{\perp}} {\rm L}^2(\re)$;
\item $\widetilde{f(H_0(\bk_{\ell})) u}(\xi,\bk_{\ell}+\bK_{\ell}) = f(\lambda(\xi,\bk_{\ell}+\bK_{\ell})) \widetilde{u}(\xi,\bk_{\ell}+\bK_{\ell})$ for any
$ (\xi,\bk_{\ell}) \in \re \times \cB$ and $\bK_{\ell} \in  \cL^{\perp}$, and any borelian function $f : \re \rightarrow \re$, where we have set
\bel{hzan4}
\lambda(\xi,\bk_{\ell}+\bK_{\ell}) := \xi^2 + | \bk_{\ell} + \bK_{\ell} |^2.
\ee 
\end{enumerate}
Notice from (a) that the following Parseval equality
\bel{hzan4b}
\| u \|_{\cH}^2 = \sum_{\bK_{\ell} \in \cL^{\perp}} \int_{\re} | \widetilde{u}(\xi,\bk_{\ell}+\bK_{\ell})|^2 {\rm d} \xi,\ \bk_{\ell} \in \cB,\ u \in \cH,
\ee
holds true, and from (b) that we have
\bel{hzan5}
\langle R_0(\bk_{\ell},z) u , v \rangle_{\cH} = \sum_{\bK_{\ell} \in \cL^{\perp}} \int_{\re} \frac{\widetilde{u}(\xi,\bk_{\ell}+\bK_{\ell}) \overline{\widetilde{v}(\xi,\bk_{\ell}+\bK_{\ell})}}{\lambda(\xi,\bk_{\ell}+\bK_{\ell}) - z} {\rm d} \xi,\ \bk_{\ell} \in \cB,\ u, v \in \cH,
\ee
for each $z \in \C$ with $\Pim{z} \neq 0$. 

Actually \eqref{hzan5} is the starting point in the derivation of the LAP for $H_0(\bk_{\ell})$, $\bk_{\ell} \in \cB_E$, stated in Proposition \ref{pr-lapz}. Its proof relies on the following a priori H\"older estimates of the generalized Fourier coefficients \eqref{hzan3} of suitably decreasing functions.
\begin{lemma}
\label{lm-holz}
Assume $\sigma >1 \slash 2$. Then, for every $u \in \cH_{\sigma} := \{ v \in \cH,\  (1+ x_1^2)^{\sigma \slash 2} v \in \cH \}$, it holds true that:
\begin{enumerate}[(a)]
\item
$| \widetilde{u}(\xi,\bk_{\ell}+\bK_{\ell}) | \leq c_{\sigma} \| u\|_{\cH_{\sigma}}$ for all $(\xi,\bk_{\ell}) \in \re \times \cB$ and $\bK_{\ell} \in \cL^{\perp}$, where we have set $c_{\sigma}:= (2 \pi)^{-1 \slash 2} \left( \int_{\re} (1+x_1^2)^{-\sigma} {\rm d} x_1 \right)^{1 \slash 2}$;
\item For each $\alpha \in [0,1] \cap  [0,\sigma-1 \slash 2)$ there exists a constant $c_{\sigma,\alpha}>0$ depending only on $\cS$,  $\sigma$ and $\alpha$, such that for all $ \xi, \xi' \in \re$, $\bk_{\ell} \in \cB$ and $\bK_{\ell} \in \cL^{\perp}$, we have:
$$ | \widetilde{u}(\xi,\bk_{\ell}+\bK_{\ell}) - \widetilde{u}(\xi',\bk_{\ell}+\bK_{\ell}) | \leq c_{\sigma,\alpha} | \xi - \xi'|^{\alpha} \| u \|_{\cH_{\sigma}}.$$
\end{enumerate}
\end{lemma}
\begin{proof}
The first claim follows readily from \eqref{hzan1}-\eqref{hzan3} and the Cauchy-Schwarz inequality. Further,
as
$${\rm e}^{-{\rm i} \xi x_1} - {\rm e}^{-{\rm i} \xi' x_1} = \int_0^1 \frac{{\rm d}}{{\rm d} t}( {\rm e}^{-{\rm i} x_1(t \xi + (1-t) \xi')} ) {\rm d} t =  -{\rm i} x_1  (\xi-\xi') \int_0^1{\rm e}^{-{\rm i} x_1(t \xi + (1-t) \xi')} {\rm d} t,$$
we have
$ | {\rm e}^{-{\rm i} \xi x_1} - {\rm e}^{-{\rm i} \xi' x_1}| \leq |x_1| | \xi-\xi' |$, whence 
\bel{hzan5star}
| {\rm e}^{-{\rm i} \xi x_1} - {\rm e}^{-{\rm i} \xi' x_1}| \leq 2^{1-\alpha} | {\rm e}^{-{\rm i} \xi x_1} - {\rm e}^{-{\rm i} \xi' x_1}|^{\alpha} \leq 2^{1-\alpha} |x_1|^{\alpha} | \xi - \xi' |^{\alpha},\ \alpha \in [0,1].
\ee
Moreover, since $({\rm e}^{-{\rm i} \xi x_1} - {\rm e}^{-{\rm i} \xi' x_1}) u(\bx) = (1+x_1^2)^{-\sigma \slash 2} ({\rm e}^{-{\rm i} \xi x_1} - {\rm e}^{-{\rm i} \xi' x_1})  (1+x_1^2)^{\sigma \slash 2} u(\bx)$, it follows from \eqref{hzan5star} that
\begin{multline*}
 \left| \int_{\re} ({\rm e}^{-{\rm i} \xi x_1} - {\rm e}^{-{\rm i} \xi' x_1}) u(x_1,\bx_{\ell}) {\rm d} x_1 \right|  
 \leq   \left(  \int_{\re} \frac{|{\rm e}^{-{\rm i} \xi x_1} - {\rm e}^{-{\rm i} \xi' x_1}|^2}{(1+x_1^2)^{\sigma}} {\rm d} x_1 \right)^{1 \slash 2}  \left(  \int_{\re} (1+x_1^2)^{\sigma} |u(x_1,\bx_{\ell})|^2  {\rm d} x_1 \right)^{1 \slash 2}  \\
 \leq  2^{1-\alpha} \left(  \int_{\re} \frac{{\rm d} x_1}{(1+x_1^2)^{\sigma-\alpha}} \right)^{1 \slash 2} | \xi - \xi' |^{\alpha} \| u(.,\bx_{\ell}) \|_{{\rm L}^{2,\sigma}(\re)},\ \bx_{\ell} \in \cS,\ \alpha \in [0, \sigma- 1 \slash 2).
\end{multline*}
From this and \eqref{hzan1}-\eqref{hzan3} then follows that
$$
| \widetilde{u}(\xi,\bk_{\ell}+\bK_{\ell}) - \widetilde{u}(\xi',\bk_{\ell}+\bK_{\ell}) | \leq  \frac{2^{1-\alpha}}{(2 \pi)^{1 \slash 2}} \left(  \int_{\re} \frac{{\rm d} x_1}{(1+x_1^2)^{\sigma-\alpha}} \right)^{1 \slash 2} | \xi - \xi'  |^{\alpha} \| u \|_{\cH_{\sigma}},
$$
by integrating over $\cS$. This terminates the proof.
\end{proof}

\subsection{LAP for $\bk_{\ell} \in \cB$}
For every  $\bk_{\ell} \in \cB$ and $\sigma > 1 \slash 2$, $u \mapsto \tilde{u}(0,\bk_{\ell})$ is  a continuous linear form on $\cH_{\sigma}$ according to \eqref{hzan1}-\eqref{hzan3} and Lemma \ref{lm-holz}(a), hence
$$
\mathcal{NH}_{\sigma}(\bk_{\ell}) := \{ u \in \cH_{\sigma},\ \tilde{u}(0,\bk_{\ell})=(2 \pi)^{-1 \slash 2} | \cS |^{-1 \slash 2} \int_{\re \times \cS} {\rm e}^{-{\rm i} \langle \bk_{\ell} , \bx_{\ell} \rangle} u(\bx) {\rm d} \bx= 0 \},$$
is a closed hyperplane of $\cH_{\sigma}$. Hence we may not regard its (topological) dual set $\mathcal{NH}_{\sigma}(\bk_{\ell})'$ as a subspace of $\cH_{\sigma}'=\cH_{-\sigma}$, in the main result of Appendix B:

\begin{pr}
\label{pr-lapz}
Let $E \in (0,E_{\delta})$ where $\delta>0$.
\begin{enumerate}[(a)]
\item For all $\bk_{\ell} \in \cB_E^+$ there exists $R_0(\bk_{\ell},E)=R_0(\bk_{\ell},E \pm {\rm i} 0):=\lim_{\varepsilon \rightarrow 0} R_0(\bk_{\ell},E +{\rm i} \varepsilon)$ in the $B(\cH)$ norm sense.
If $\bk_{\ell} \in \cB_E^-$ (resp. $\bk_{\ell} \in \cB \backslash \cB_E$) there is $R_0(\bk_{\ell},E \pm {\rm i} 0):=\lim_{\varepsilon \downarrow 0} R_0(\bk_{\ell},E \pm {\rm i} \varepsilon)$ in the $B(\cH_{\sigma},\cH_{-\sigma})$ norm sense (resp. the $B(\mathcal{NH}_{\sigma}(\bk_{\ell}),\mathcal{NH}_{\sigma}(\bk_{\ell})')$ norm sense) provided
$\sigma > 1 \slash 2$ (resp. $\sigma > 1$).
\item For all $\bk_{\ell} \in \cB_E^+$ we have $\lim_{\varepsilon \downarrow 0} \varepsilon R_0(\bk_{\ell},E \pm {\rm i} \varepsilon)=0$ in the $B(\cH)$ norm sense. If $\bk_{\ell} \in \cB \backslash \cB_E^+$ then for every
$u \in \cH$ it holds true that
$\lim_{\varepsilon \downarrow 0} \varepsilon R_0(\bk_{\ell},E \pm {\rm i} \varepsilon) u =0$  in $\cH$ weakly.
\end{enumerate}
\end{pr}

\begin{proof} 
We prove the successively the claims (a) and (b).\\ 
(a) For $\bk_{\ell} \in \cB_E^+$, the result follows readily from the first resolvent formula, since $\sigma( H_0(\bk_{\ell}) ) = [ | \bk_{\ell} |^2 , +\infty)$. Therefore it is sufficient to examine the case $\bk_{\ell} \in \cB \backslash \cB_E^+$.
We shall actually prove that the function $z \mapsto R_0(\bk_{\ell},z)$ is uniformly continuous in $\C^{\pm} \cap \cK_E$, 
for some appropriate compact neighborhood $\cK_E$ of $E$ in $\C$, where $\C^{\pm}:=\{ \zeta \in \C,\ \pm \Pim{\zeta} >0 \}$.
To this purpose we fix
$u, v \in \cH$, set $h(\xi,\bk_{\ell}+\bK_{\ell}) := \tilde{u}(\xi,\bk_{\ell}+\bK_{\ell}) \overline{\tilde{v}(\xi,\bk_{\ell}+\bK_{\ell})}$ for all $\bK_{\ell} \in \cL^{\perp}$, and, with reference to \eqref{hzan4} and \eqref{hzan4}, introduce the integrals
\bel{hzan6}
r_{\bK_{\ell}}(z):= \frac{h(\xi,\bk_{\ell}+\bK_{\ell})}{\lambda(\xi,\bk_{\ell}+\bK_{\ell})-z} {\rm d} \xi,\ z \in \C^{\pm}.
\ee 
Let $\cK_E := \{ z \in \C,\ | \Pre{z} - E | \leq d\ {\rm and}\ | \Pim{z} | \leq 1 \}$ where
$$ d:=\left\{ \begin{array}{ll} 
\delta E_{\delta} \slash 2 & {\rm if}\ \bk_{\ell} \in \cB \backslash \cB_E \\
\min((E - | \bk_{\ell} |^2) \slash 4, \delta E_{\delta} \slash 2) & {\rm if}\ \bk_{\ell} \in \cB_E^-. \end{array} \right. $$
We have $| \lambda(\xi,\bk_{\ell}+\bK_{\ell}) - z | \geq 2 d$ for every $\xi \in \re$, $\bK_{\ell} \in \cL^{\perp} \backslash \{ 0 \}$
and $z \in \cK_E$ by \eqref{r6b}, hence
\bel{hzan7}
\left| \sum_{\bK_{\ell} \in \cL^{\perp} \backslash \{ 0 \} } (r_{\bK_{\ell}}(z') - r_{\bK_{\ell}}(z)) \right| \leq \frac{\| u \|_{\cH} \| v \|_{\cH}}{4d^2},\ z, z' \in \cK_E,
\ee
from \eqref{hzan4b}. 
Thus we are left with the task of examining the behaviour for $z \in \C^{\pm} \cap \cK_E$ of 
$$r_0(z)= \sum_{\zeta=+,-} r_{0,\zeta}(z)\ {\rm where}\ r_{0,\zeta}(z):= \int_0^{+\infty} \frac{h(\zeta \xi,\bk_{\ell})}{\lambda(\xi,\bk_{\ell})-z} {\rm d} \xi,\ \zeta=+,-.$$
We treat the the two cases $\bk_{\ell} \in \cB_E^-$ and $\bk_{\ell} \in \cB \backslash \cB_E$ separately.\\
1. We start with $\bk_{\ell} \in \cB_E^-$. Setting $\xi_0=\xi_0(\bk_{\ell}):=(E - | \bk_{\ell}|^2)^{1 \slash 2}$, we
introduce a function $\chi \in {\rm C}^1(\re_+;[0,1])$ verifying 
$$ \chi(\xi)= \left\{ \begin{array}{cl} 1 & {\rm if}\ \xi \in [(\xi_0^2-2d)^{1 \slash 2}, (\xi_0^2+2d)^{1 \slash 2}] \\ 0 & {\rm if}\ \xi \in \re_+ \backslash ((\xi_0^2-3d)^{1 \slash 2}, (\xi_0^2+3d)^{1 \slash 2}),
\end{array} \right. $$
and decompose $r_{0,\zeta}(z)$, $\zeta=+,-$, into the sum: 
\bea
r_{0,\zeta}(z) & = &  \int_0^{+\infty} \frac{(1 - \chi(\xi)) h(\zeta \xi,\bk_{\ell})}{\lambda(\xi,\bk_{\ell})-z} {\rm d} \xi + \int_0^{+\infty} \frac{\chi(\xi) h(\zeta \xi,\bk_{\ell})}{\lambda(\xi,\bk_{\ell})-z} {\rm d} \xi \nonumber \\
& := & \mathfrak{a}_{\zeta}(z) + \mathfrak{b}_{\zeta}(z). \label{hzan8}
\eea
Taking into account that $| \xi^2 - (z-| \bk_{\ell}|^2) | \geq d>0$ for every $\xi \in {\rm supp}(1-\chi)$ and $z \in \cK_E$, we deduce from \eqref{hzan4b} that
\bel{hzan8b}
| \mathfrak{a}_{\zeta}(z') - \mathfrak{a}_{\zeta}(z) | \leq \frac{|z'-z|}{d^2} \| u \|_{\cH} \| v \|_{\cH},\ z,z' \in \cK_E,\ \zeta =+,-.
\ee
Further, the remaining term $\mathfrak{b}_{\zeta}(z)$ is brought into the form
$$ \mathfrak{b}_{\zeta}(z) = \int_{I} \frac{g_{\zeta}(\lambda)}{\lambda - (z - | \bk_{\ell} |^2)} {\rm d} \lambda\ {\rm where}\ 
g_{\zeta}(\lambda):=\frac{\chi(\lambda^{1 \slash 2}) h(\zeta \lambda^{1 \slash 2},\bk_{\ell})}{2 \lambda^{1 \slash 2}}, $$
and $I:=(\xi_0^2-3d, \xi_0^2+3d)$, 
by performing the change of variable $\lambda=\xi^2$ in the second integral of \eqref{hzan7}. Bearing in mind that $\overline{I}$ is at distance $d>0$ from 0, and then applying Lemma \ref{lm-holz}(b), we get that $g_{\zeta}$ is H\"older continuous in $\overline{I}$. Namely, $\alpha$ being fixed in $[0,1] \cap (0,\sigma-1 \slash 2)$, we may find 
a constant $A_g>0$, independent of $u$ and $v$, such that we have
$$ | g_{\zeta}(\lambda') - g_{\zeta}(\lambda) | \leq A_g |\lambda' - \lambda|^{\alpha} \| u \|_{\cH_{\sigma}} \| v \|_{\cH_{\sigma}},\ 
\lambda, \lambda' \in \overline{I},\ \zeta = +,-, $$
It follows from this, the identities $g_{\zeta}(\xi_0^2 \pm 3d)=0$, and the Plemelj-Privalov theorem (see \cite{Mu}[Part 1, Chap. 2, \S 22]) that $\mathfrak{b}_{\zeta}$ is extendable to an $\alpha$-H\"older continuous function, also denoted by $\mathfrak{b}_{\zeta}$, in $V^{\pm}:=\{ z \in \overline{\C^{\pm}},\ \Pre{z} \in \overline{I} \}$: there exists $c \in {\rm C}^0((V^{\pm})^2;\re_+)$ satisfying
\bel{hzan8c}
| \mathfrak{b}_{\zeta}(z')-\mathfrak{b}_{\zeta}(z) | \leq c(z,z') |z'-z|^{\alpha} \| u \|_{\cH_{\sigma}} \| v \|_{\cH_{\sigma}},\ z, z' \in V^{\pm}.
\ee
Now, putting \eqref{hzan5} and \eqref{hzan6}-\eqref{hzan8c} together, we end up getting that
$$ | \langle (R_0(\bk_{\ell},z') -  R_0(\bk_{\ell},z)) u , v \rangle_{\cH} | \leq C | z' - z|^{\alpha} \| u \|_{\cH_{\sigma}} \| v \|_{\cH_{\sigma}},\ z , z' \in \overline{\C^{\pm}} \cap \cK_E, $$
for some constant $C$ independent of $u$ and $v$, since $\overline{\C^{\pm}} \cap \cK_E$ is a compact subset of $V^{\pm}$. This yields
$$ \| R_0(\bk_{\ell},z') -  R_0(\bk_{\ell},z)  \|_{\cB(\cH_{\sigma},\cH_{-\sigma})} \leq  C | z' - z|^{\alpha},\
z, z' \in \overline{\C^{\pm}} \cap \cK_E, $$
hence the result.\\
2. Let us now consider the case where $\bk_{\ell} \in \cB \backslash \cB_E$. We define 
$\mathfrak{a}_{\zeta}$ and
$\mathfrak{b}_{\zeta}$ as in \eqref{hzan8}, where $\chi \in {\rm C}^1(\re_+;[0,1])$ satisfies
$$ \chi(\xi)= \left\{ \begin{array}{cl} 1 & {\rm if}\ \xi \in [(0, (2d)^{1 \slash 2}] \\ 0 & {\rm if}\ \xi \in \re_+ \backslash (0, (3d)^{1 \slash 2}).
\end{array} \right. $$
As $|\xi^2-(z-E)| \geq d >0$ for every $\xi \in {\rm supp}(1 - \chi)$ and $z \in \cK_E$ by standard computations, it is easy to check that \eqref{hzan8b} still holds true for $|\bk_{\ell}|^2=E$. Similarly, we find that
$\mathfrak{b}_{\zeta}(z) = \int_{I} g_{\zeta}(\lambda) \slash (\lambda - (z - E)) {\rm d} \lambda$ where $I=(0,3d)$ and
$g_{\zeta}(\lambda)$ is unchanged. The next step of the proof involves choosing $\alpha \in [0,1] \cap (1 \slash 2,\sigma- 1 \slash 2)$ and recalling that $\tilde{u}(0,\xi)=\tilde{v}(0,\xi)=0$, in such a way that
$$ |h(\zeta (\lambda)^{1 \slash 2},\bk_{\ell}) | \leq c_{\sigma,\alpha}^2 \| u \|_{\cH_{\sigma}} \| v \|_{\cH_{\sigma}} 
\lambda^{\alpha},\ \lambda \in \re_+,\ \zeta = +,-,$$
according to Lemma \ref{lm-holz}(b). From this then follows that $\xi \mapsto g_{\zeta}(\xi,\bk_{\ell})$, $\zeta=+,-$, can be extended to an $(\alpha-1 \slash 2)$-H\"older continuous function in $\overline{I}$, verifying $g_{\zeta}(3d)=0$.
Arguing as before we thus get that
$$ \| R_0(\bk_{\ell},z') -  R_0(\bk_{\ell},z)  \|_{\cB(\mathcal{NH}_{\sigma}(\bk_{\ell}),\mathcal{NH}_{\sigma}(\bk_{\ell})')} \leq  C | z' - z|^{\alpha-1 \slash 2},\
z, z' \in \overline{\C^{\pm}} \cap \cK_E, $$
for some constant $C>0$, which yields the result.\\
(b) The claim (b) is an immediate consequence of (a) for $\bk_{\ell} \in \cB_E^+$. If $\bk_{\ell} \in \cB \backslash \cB_E^+$, the result follows from the dominated convergence theorem as we have $\lim_{\varepsilon \rightarrow 0}  \varepsilon  h (\xi,\bk_{\ell}) \slash (\lambda(\xi,\bk_{\ell})- E - {\rm i} \varepsilon) = 0$ for a.e. $\xi \in \re$, and 
$| \varepsilon h (\xi,\bk_{\ell}) | \slash | \lambda(\xi,\bk_{\ell})- E - {\rm i} \varepsilon) | \leq | h ( \xi,\bk_{\ell}) |$ for all $\varepsilon \in \re$ and a.e. $\xi \in \re$, with $\int_{\re} | h (\xi,\bk_{\ell}) | {\rm d} \xi \leq \| u \|_{\cH} \| v \|_{\cH}$.
\end{proof}

\begin{remark}
\label{rmk-lapz}
The result of Proposition \ref{pr-lapz}(a) for $\bk_{\ell} \in \cB_E$ can be recovered from the reasonning developped in \cite{Agmon}[\S 4] (see also \cite{DG}, \cite{wed1} and \cite{wed2}) but this would require that the estimate \eqref{nz0} in  Lemma \ref{lm-nz} be generalized to the case of $\bk_{\ell} \in \cB_{E}^-$. The derivation of this particular result being quite tricky we prefer to apply the above method essentially based on the Plemelj-Privalov theorem, which is a very powerful tool in this framework.

\end{remark}

\subsection{More on $R_0(\bk_{\ell},z )$ for $(\bk_{\ell},z) \in \cB_E^+ \times \C$ such that $| \Pre{z}| < | \bk_{\ell} |^2$}

In view of characterizing the rate of decay of guided states associated to $\bk_{\ell} \in \mathcal{C}_g(E)$ in the direction orthogonal to $x_1$ (see Theorem \ref{thm-sw} and the proof of Lemma \ref{lm-sw3}), we need the following result about the resolvent of $H_0(\bk_{\ell})$ for $\bk_{\ell} \in \cB_E^+$:

\begin{lemma}
\label{lm-nz}
Let $\bk_{\ell} \in \cB_E^+$ and $z \in \C$ be such that $| \Pre{z}| < | \bk_{\ell} |^2$ and $| \Pim{z} | \leq 1$. Then $R_0(\bk_{\ell},z) v \in {\rm H}^{2,\sigma}(\re \times \cS):={\rm H}^{2}(\re \times \cS; (1+x_1^2)^{\sigma \slash 2} {\rm d} \bx)$ for every $v \in \cH_{\sigma}$, $\sigma \geq 0$, and it holds true that
\bel{nz0}
c_1 \| v \|_{\cH_{\sigma}} \leq \| R_0(\bk_{\ell},z ) v \|_{{\rm H}^{2,\sigma}(\re \times \cS)} \leq c_2 \| v \|_{\cH_{\sigma}},
\ee
where $c_1>0$ and $c_2>0$ are two constants independent of $v$ and $\Pim{z}$.
\end{lemma}
\begin{proof}
Bearing in mind that the domain of $H_0(\bk_{\ell})$ is a subset of ${\rm H}^{2}(\re \times \cS)$ according to \eqref{s6b} and that  $\sigma(H_0(\bk_{\ell}))=[| \bk_{\ell}|^2,+\infty)$, we may write
$$ | \bk_{\ell} |^2 \| u \|_{\cH}^2 \leq \langle H_0(\bk_{\ell}) u , u \rangle_{\cH} =  -\langle \Delta u , u \rangle_{\cH} = \| \nabla u \|_{\cH}^2 \leq \| \Delta u \|_{\cH} \| u \|_{\cH}, $$
for every $u \in {\rm dom}\ H_0(\bk_{\ell})$, from where we get that
\bel{nz1}
| \bk_{\ell} | \| u \|_{\cH} \leq \| \nabla u \|_{\cH} \leq | \bk_{\ell} |^{-1}  \| \Delta u \|_{\cH}.
\ee
Next, setting $\varepsilon:=(| \bk_{\ell}|^2- | \Pre{z}|) \slash (1 + | \bk_{\ell} |^2) >0$, in such a way that we have $\varepsilon = (1-\varepsilon) | \bk_{\ell}|^2 -  |\Pre{z}|$, and then noticing from \eqref{nz1} that
\beas
 \| (H_0(\bk_{\ell}) - z) u \|_{\cH} & \geq & \| (H_0(\bk_{\ell}) -\Pre{z}) u \|_{\cH} \\
& \geq &  \varepsilon \| \Delta u \|_{\cH} + (1-\varepsilon) \| \Delta u \|_{\cH} - | \Pre{z} | \| u \|_{\cH} \\
& \geq &  \varepsilon \| \Delta u \|_{\cH} + ((1-\varepsilon) | \bk_{\ell}|^2 - | \Pre{z} | ) \| u \|_{\cH},
\eeas
we find that
$ \| (H_0(\bk_{\ell}) - z) u \|_{\cH} \geq \varepsilon (  \| \Delta u \|_{\cH} + \| u \|_{\cH} )$. This entails
\bel{nz2}
C_1 \| u \|_{{\rm H}^{2}(\re \times \cS)} \leq \| (H_0(\bk_{\ell}) - z ) u \|_{\cH} \leq C_2 \| u \|_{{\rm H}^{2}(\re \times \cS)},\ u \in {\rm dom}\ H_0(\bk_{\ell}), 
\ee
for some constants $C_1 >0$ and $C_2>0$ depending only on $|\bk_{\ell}|$ and $| \Pre{z}|$. Further, since 
$ (1+ x_1^2)^{\sigma \slash 2} u \in {\rm dom}\ H_0(\bk_{\ell}) $ for all $u \in {\rm dom}\ H_0(\bk_{\ell})$, and $\| u \|_{{\rm H}^{2,\sigma}(\re \times \cS)}$ is equivalent to $\|  (1+ x_1^2)^{\sigma \slash 2} u \|_{{\rm H}^{2}(\re \times \cS)}$, it follows from \eqref{nz2} that $\| u \|_{{\rm H}^{2,\sigma}(\re \times \cS)}$ is equivalent to $\| (H_0(\bk_{\ell}) - z ) (1+ x_1^2)^{\sigma \slash 2} u \|_{\cH}$. As $\| (H_0(\bk_{\ell}) - z ) (1+ x_1^2)^{\sigma \slash 2} u \|_{\cH}$ is easily seen to be
equivalent to $\| (H_0(\bk_{\ell}) - z ) u \|_{\cH_{\sigma}}$, we finally obtain the result from this by taking $u= R_0(\bk_{\ell},z) v$. 
\end{proof}

We now conclude Appendix B with the two following comments.
\begin{enumerate}[(a)]
\item It is not hard to see that a LAP for $H_0=\int_{\bk_{\ell} \in \cB} H_0(\bk_{\ell})$, similar to the one stated in \cite{Agmon}, can be obtained from Proposition \ref{pr-lapz} and Lemma \ref{lm-nz}, by ``integrating" the above results w.r.t. $\bk_{\ell}$ over $\cB$. The technical difficulty arising in this method when $| \bk_{\ell} | = E^{1 \slash 2}$ is easily overcomed by the change of
the integration variables $(\xi,\bk_{\ell})$ into spherical coordinates, in a neighborhood of the sphere $S(0,E^{1 \slash 2}) \subset \mathbb{R}\times\cS$.
\item Furthermore, the results of Theorem \ref{thm-a} allow for the derivation of a LAP for the perturbed operator $H$, under slightly less rectrictive conditions than in \cite{filonov1}-\cite{filErratum} or \cite{Ger}. Nevertheless, as already mentioned in the introduction, the proof of this result is left to the reader in order to avoid the inadequate expense of the size of this paper. 
\end{enumerate}

\noindent {\bf Acknowledgement}\\
\noindent The authors are grateful to S. Rigat (Universit\'e d'Aix-Marseille) for fruitful discussions on weak analyticity.

\end{document}